\documentclass[11pt,reqno]{amsart}
\usepackage{latexsym,amssymb,graphicx,multicol,mathrsfs,color,xypic,amscd, amsxtra, verbatim}

\theoremstyle{plain}
\newtheorem*{claim}{Claim}
\newtheorem*{CaseI}{Case I}
\newtheorem*{CaseII}{Case II}

\newtheorem*{reductionI}{Reduction 1}
\newtheorem*{reductionII}{Reduction 2}
\newtheorem*{reductionIII}{Reduction 3}

\newtheorem{theorem}{Theorem}[section]
\newtheorem*{thm}{Theorem}
\newtheorem{proposition}[theorem]{Proposition}
\newtheorem{lemma}[theorem]{Lemma}

\newtheorem{corollary}[theorem]{Corollary}

\theoremstyle{definition}
\newtheorem{remark}[theorem]{Remark}

\newtheorem*{remarks}{Remarks}

\newtheorem*{mainresult}{Main Result}
\newtheorem{definition}[theorem]{Definition}
\newtheorem*{defn}{Definition}
\newtheorem{example}[theorem]{Example}

\theoremstyle{remark}
\newtheorem*{Open Question}{Open Question}

\input xypic
\xyoption{all}

\begin{document}
\def\A{\text{A}}
\def\Aut{\text{Aut\,}}
\def\Ann{\text{Ann\,}}
\def\Bl{\text{\rm Bl}}
\def\codim{\text{\rm codim}}
\def\Gal{\text{Gal\,}}
\def\tphi{\tilde{\phi}}
\def\dim{\text{dim\,}}
\def\characteristic{\text{characteristic\,}}
\def\discrep{\text{discrep}}
\def\Def{\text{Def\,}}
\def\LDef{\overline{\text{Def}}\,}
\def\Tbar{\underline{T}}
\def\length{\text{\rm length}}
\def\B{\mathcal{B}}
\def\C{\mathcal{C}}
\def\D{\mathcal{D}}
\def\div{\text{\rm div}}
\def\e{\epsilon}
\def\ev{\text{\rm ev}}
\def\E{\mathcal{E}}
\def\Eq{\text{Eq\,}}
\def\Exc{\text{Exc\,}}
\def\Reg{\text{Reg\,}}
\def\Eff{\text{Eff\,}}
\def\Ext{\mathscr{E}\!xt}
\def\F{\mathscr{F}}
\def\Frac{\text{Frac\,}}
\def\G{\mathscr{G}}
\def\H{\mathscr{H}}
\def\Hom{\mathscr{H}\!om}
\def\hom{\text{\rm Hom}}
\def\I{\mathscr{I}}
\def\Ind{\text{Indet\,}}
\def\Image{\text{\rm Image}}
\def\Isom{\text{\rm Isom}}
\def\id{\text{\rm id}}
\def\im{\text{im\,}}
\def\J{\mathscr{J}}
\def\K{\mathscr{K}}
\def\Ker{\text{\rm Ker}}
\def\L{\mathscr{L}}
\def\Moduli{\text{\underline{Attaching Moduli}}}
\def\Maps{\text{\underline{Attaching Maps}}}
\newcommand{\Mod}[2][n]{\ensuremath{\overline{\mathcal{M}}_{1,#1}(#2)}}
\newcommand{\m}{\ensuremath{m}}
\def\NE{\overline{\text{NE}}}
\def\Nef{\text{Nef}}
\def\O{\mathscr{O}}
\def\R{\mathcal{R}}
\def\T{\mathcal{T}}
\def\ttau{\tilde{\tau}}
\def\P{\mathbb{P}}
\def\pn{\{p_i\}_{i=1}^{n}}
\def\Q{\mathbb{Q}}
\def\W{\mathcal{W}}
\def\X{\mathcal{X}}
\def\Y{\mathcal{Y}}
\def\U{\mathcal{U}}
\def\V{\mathcal{V}}
\def\Z{\mathcal{Z}}
\def\pic{\text{\rm Pic}\,}
\def\Sch{\text{Sch}}
\def\SA{(\mathcal{S}, \mathcal{A})}
\def\red{_{\text{\rm red}}}
\def\Res{\text{\rm Res}}
\def\Spec{\text{\rm Spec\,}}
\def\Proj{\text{\rm Proj\,}}
\def\Supp{\text{\rm Supp\,}}
\def\Mor{\text{\rm Mor}}
\def\Pic{\text{\rm Pic\,}}
\def\Ver{\text{\rm Ver}}
\def\ql{\{q_i\}_{i=1}^{l}}
\def\qm{\{q_i\}_{i=1}^{m}}
\def\qmp{\{q_i'\}_{i=1}^{m}}
\def\pn{\{p_i\}_{i=1}^{n}}
\def\sigman{\{\sigma_i\}_{i=1}^{n}}
\def\taum{\{\tau_i\}_{i=1}^{m}}
\def\taul{\{\tau_i\}_{i=1}^{l}}
\newcommand{\sigmav}[1]{\{\sigma_i\}_{i=1}^{#1}}
\newcommand{\pv}[1]{\{p_i\}_{i=1}^{#1}}
\newcommand{\tauv}[1]{\{\tau_i\}_{i=1}^{#1}}
\newcommand{\qv}[1]{\{q_i\}_{i=1}^{#1}}
\def\S{\overline{\mathcal{M}}}
\def\M{\overline{M}}
\newcommand{\SV}[2]{\overline{\mathcal{M}}_{1,#1}(#2)}
\newcommand{\NS}[2]{\overline{\mathcal{M}}_{1,#1}(#2)^*}
\newcommand{\MV}[2]{\widetilde{\mathcal{M}}_{1,#1}(#2)}
\newcommand{\NM}[2]{\overline{M}_{1,#1}(#2)^*}

\title[Modular compactifications]{Modular compactifications of\\ the space of pointed elliptic curves II}
\author{David Ishii Smyth}
\address{Department of Mathematics\\ Harvard University\\ Cambridge, MA 02138}
\email{dsmyth@math.harvard.edu}

\maketitle
\begin{abstract} We prove that the moduli spaces of $n$-pointed $m$-stable curves introduced in our previous paper have projective coarse moduli. We use the resulting spaces to run an analogue of Hassett's log minimal model program for $\M_{1,n}$.
\end{abstract}

\tableofcontents
\pagebreak

\section{Introduction}
In \cite{Hgenus2}, Hassett proposed the problem of studying log canonical models of $\M_{g}$. For any $\alpha \in \Q \cap [0,1]$ such that $K_{\S_{g}}+\alpha\Delta$ is big, Hassett and Keel define
\[
\M_{g}(\alpha):=\Proj \oplus_{m \geq0} H^0(\S_{g}, m(K_{\S_{g}}+\alpha\Delta)),
\]
where the sum ranges over sufficiently divisible $m$, and ask whether the spaces $\M_{g}(\alpha)$ admit a modular interpretation. In this paper, we consider an analogous problem for $\M_{1,n}$. For any $s \in \Q$, we define
\begin{align*}
&D(s):=s\lambda+\psi-\Delta, \\
&R(s):=\oplus_{m \geq 0}H^0(\S_{1,n}, mD(s)),\\
&\M_{1,n}^s:=\Proj R(s),
\end{align*}
where $\lambda$, $\psi$, and $\Delta$ are certain tautological divisor classes on $\S_{1,n}$ (these will be defined in Section 3), and the sum defining $R(s)$ is taken over $m$ sufficiently divisible. We will show that the section ring $R(s)$ is finitely-generated and that the associated birational model $\M_{1,n}^s$ admits a modular interpretation for all $s \in \Q$ such that $D(s)$ is big. In fact, the birational models arising in this construction are precisely the moduli spaces of $m$-stable curves introduced in \cite{SmythEI}.

In \cite[Theorem 3.8]{SmythEI}, we proved that the moduli stack of $n$-pointed $m$-stable curves is an irreducible, proper, Deligne-Mumford stack over $\Spec \mathbb{Z}[1/6]$. In this paper, we work over a fixed algebraically closed field $k$ of characteristic zero. Henceforth, $\S_{1,n}(m)$ will denote the moduli stack of $m$-stable curves over $k$,  $\M_{1,n}(m)$ the corresponding coarse moduli space, and $\M_{1,n}(m)^*$  the normalization of the coarse moduli space. Our main result (Corollary \ref{C:MainResult}) is:

\begin{mainresult} Given $s \in \Q$ and $m, n \in \mathbb{N}$ satisfying $m<n$, we have
\begin{enumerate}
\item[]
\item $D(s)$ is big iff $s \in (12-n, \infty)$\\
\item $\M_{1,n}^s=
\begin{cases}
\M_{1,n} &\text{ iff } s \in (11, \infty)\\
\M_{1,n}(1) & \text{ iff } s \in (10,11]\\
\M_{1,n}(m)^* &\text{ iff } s \in (11-m,12-m)\text{ and $m \in \{2, \ldots, n-2\}$}\\
\M_{1,n}(n-1)^* & \text{ iff } s \in (12-n,13-n]\\
\end{cases}
$
\end{enumerate}
\end{mainresult}
\begin{remarks}
\begin{enumerate}
\item[]
\item Note that we do not give a modular interpretation of the model $\M_{1,n}^s$ for the transitional values $s=10, 9, \ldots, 14-n$. At these values, the model $\M_{1,n}^s$ may be viewed as the intermediate small contraction associated to the flip $\M_{1,n}(m-1)^* \dashrightarrow \M_{1,n}(m)^*$.
\item We will show that $\S_{1,n}(m)$ is a smooth stack iff $m \leq 5$ (Corollary \ref{C:Smoothness} and Corollary \ref{C:Singular}). In particular, $\M_{1,n}(m)=\M_{1,n}(m)^*$ for $m \leq 5$. We do not know whether $\M_{1,n}(m)=\M_{1,n}(m)^*$ for $m \geq 6$.
\end{enumerate}
\end{remarks}
Our main result gives a complete Mori chamber decomposition of the two-dimensional slice of the effective cone of $\M_{1,n}$ spanned by $\lambda$ and $\psi-\Delta$ (Figure 1). Now let us explain how this result is connected to the log minimal model program for $\M_{g}$. Recall that the canonical divisor of $\S_{g}$ is given by
$K_{\S_{g}} = 13\lambda-2\Delta \in \Pic(\S_{g})$, so $K_{\S_{g}}+\alpha\Delta$ is numerically proportional to a uniquely defined divisor of the form $s\lambda-\Delta$, where $s$ is the \emph{slope} of the divisor. We may define
\begin{align*}
&D(s):=s\lambda-\Delta, \\
&R(s):=\oplus_{m \geq 0}H^0(\S_{g}, mD(s)),\\
&\M_{g}^s:=\Proj R(s).
\end{align*}
We have $\M_{g}(\alpha)=\M_{g}^s$ for $s=\frac{13}{2-\alpha}$ so describing the birational models $\M_{g}(\alpha)$ is equivalent to describing the models $\M_{g}^s$.
In this notation, results of Hassett and Hyeon \cite{HH1, HH2} give
\begin{thm}[Hassett-Hyeon]
\begin{align*}
\M_{g}(s)=\begin{cases}
&\M_{g} \text{ if }s \in (11, \infty)\\
&\M_{g}^s=\M_{g}^{ps}  \text{ if } s \in (10,11]\\
&\M_{g}^s=\M_{g}^{qs}   \text{ if } s \in (10-\epsilon, 10),
\end{cases}
\end{align*}
where $\M_{g}^{ps}$ is the moduli space of pseudostable curves (in which elliptic tails are replaced by cusps) and $\M_{g}^{qs}$ is the moduli space of quasistable curves (in which elliptic tails and bridges are replaced by cusps and tacnodes).
\end{thm}

\begin{figure}\label{F:slope2}
\scalebox{.75}{\includegraphics{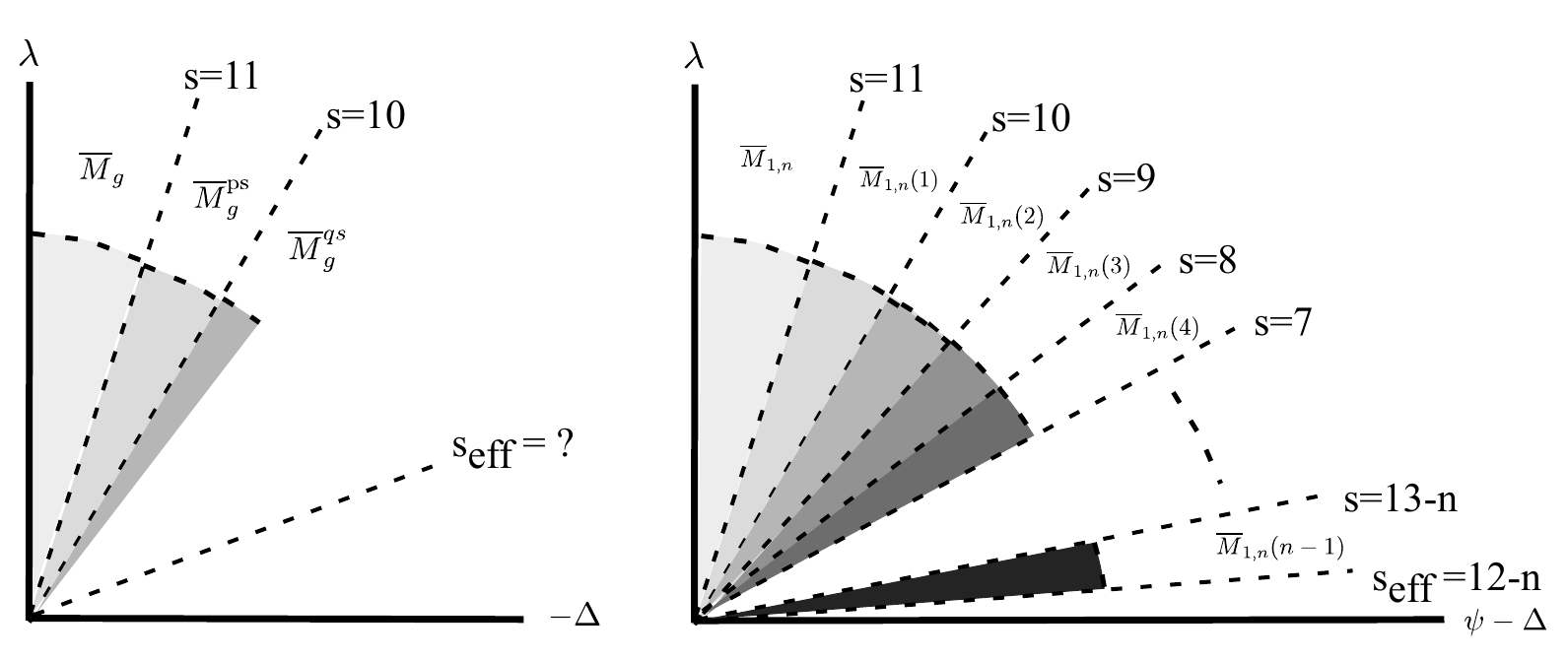}}\label{F:slope2}
\caption{Comparison of log minimal model program for $\M_{g}$ and $\M_{1,n}$.}
\end{figure}

Our results for $\M_{1,n}$ are connected to the log minimal model program for $\M_{g}$ by the following observation: For $g>>0$, we may define a closed immersion
$$
i: \S_{1,n} \hookrightarrow \S_{g},
$$
by gluing fixed tails of genus $g_1, \ldots, g_n$ (satisfying $g_1+ \ldots +g_n+1=g$) onto the $n$ marked points. One easily checks that the restriction of the divisor $s\lambda-\Delta$ on $\S_{g}$ to the subvariety $i(\S_{1,n})$ is simply $s\lambda+\psi-\Delta$, i.e. $i^*D(s)=D(s)$. Thus, \emph{our results track the effect of Hassett-Keel log minimal model program on $\M_{1,n}$, viewed as a subvariety of $\M_{g}$}. In our view, the fact that every birational model $\M_{1,n}^s$ admits a modular interpretation gives strong evidence that the models $\M_{g}^s$ should admit a modular interpretation. Furthermore, our results suggest that elliptic $m$-fold points should arise in the moduli problem associated to $\M_{g}^{s}$ at slope $s=12-m$.

Finally, we should remark that our main result can also be formulated as running a log minimal model program on $\M_{1,n}$ provided one scales $\Delta_{irr}$ rather than $\Delta$. Here, $\Delta_{irr}$ denotes the irreducible component of the boundary  whose generic point parametrizes an irreducible curve, and we set $\Delta_{red}:=\overline{\Delta \backslash \Delta_{irr}}$. Using the relations in $\Pic(\S_{1,n})$ (Proposition \ref{P:StartingPicardGroup}), one easily checks that
$$
s\lambda+\psi-\Delta \equiv K_{\S_{1,n}}+\alpha\Delta_{irr}+\Delta_{red} \text{ iff } \alpha=\frac{s-1}{12}
$$
Thus, our main result is equivalent to the statement
\begin{align*}
\Proj \oplus_{m \geq0}\Gamma(\S_{1,n}, m(K_{\S_{1,n}}+\alpha\Delta_{irr}+\Delta_{red}) )=
\begin{cases}
\M_{1,n}&\text{ iff }\alpha \in (5/6, \infty)\\
\M_{1,n}(1)&\text{ iff }\alpha \in (3/4,5/6]\\
\M_{1,n}(m)^*&\text{ iff }\alpha \in (\frac{10-m}{12}, \frac{11-m}{12})\\
\M_{1,n}(n-1)^*&\text{ iff }\alpha \in (\frac{11-n}{12},\frac{12-n}{12}]\\
\end{cases}
\end{align*}
Note that $\alpha$ becomes negative when $m \geq 11$, so that the birational models $\M_{1,n}(m)^*$ are only log canonical models for $m \leq 10$. An amusing consequence of this result is that the normalization of a versal deformation space for an elliptic $m$-fold point has log canonical singularities for $m \leq 10$. As far as we know, there is no proof of this fact by means of pure deformation theory.

It is natural to ask whether the log canonical models $\Proj \oplus_{m \geq0}\Gamma(\S_{1,n}, m(K_{\S_{1,n}}+\alpha\Delta))$ can be given a modular interpretation. In forthcoming work, we will extend our main result by considering
\begin{align*}
&D(s,t):=s\lambda+t\psi-\Delta, \\
&R(s,t):=\oplus_{m \geq 0}H^0(\S_{1,n}, mD(s,t)),\\
&\M_{1,n}^{s,t}:=\Proj R(s,t).
\end{align*}
We will show that each birational model $\M_{1,n}^{s,t}$ is isomorphic to the normalization of one of the moduli spaces of $(m,\mathcal{A})$-stable curves $\M_{1,\mathcal{A}}(m)$ introduced in \cite{SmythEI}. It is easy to see that $K_{\S_{1,n}}+\alpha\Delta$ is numerically equivalent to a divisor of the form $D(s,t)$, so we obtain an affirmative answer to the preceding question.

\subsection{Notation}\label{S:Notation}
Throughout this paper, we work over a fixed algebraically closed field $k$ of characteristic zero. An $n$-pointed curve $(C, \{p_i\}_{i=1}^{n})$ is a reduced, connected, one-dimensional scheme of finite type over $k$ with $n$ distinct smooth points $p_1, \ldots, p_n \in C$. A family of $n$-pointed curves $(f:\C \rightarrow T, \sigman)$ is a flat, proper morphism $\C \rightarrow T$ with $n$ sections $\sigman$, whose geometric fibers are $n$-pointed curves. We will frequently refer to definitions introduced in our earlier paper \cite{SmythEI}. In particular, we assume the reader 
is familiar with the definition of an \emph{elliptic $m$-fold point} \cite[Definition 2.1]{SmythEI} and an \emph{$n$-pointed $m$-stable curve} \cite[Definition 3.7]{SmythEI} .
\subsection{Outline of paper}
In this section, we outline the contents of this paper. In Section \ref{S:Geometry}, we study the stratification of $\S_{1,n}(m)$ by singularity type, i.e. the stratification
 $$\S_{1,n}(m)=\mathcal{M}_{1,n} \coprod \E_0 \coprod \E_1 \coprod \ldots \coprod \E_{m},$$
where $\E_{0}$ is the locus of singular curves with only nodal singularities, and $\E_{l}$ $(l \geq 1)$ is the locus of curves with an elliptic $l$-fold point. In Section 2.1, we use deformation theory to analyze local properties of $\S_{1,n}(m)$ and the individual strata $\E_{l}$. In Section \ref{S:AttachingModuli}, we study the ``moduli of attaching data" of the elliptic $m$-fold point. We show that isomorphism classes of elliptic $m$-fold pointed curves with given pointed normalization $(\tilde{C}, \{q_i\}_{i=1}^{m})$ are naturally parameterized by $(k^{*})^{m-1}$. In Section 2.3, we construct a modular compactification $(k^{*})^{m-1} \subset \P^{m-1}$ by considering all isomorphism classes of elliptic $m$-fold pointed curves whose normalization is obtained from the given $(\tilde{C}, \{q_i\}_{i=1}^{m})$ by sprouting semistable $\P^{1}$'s along a proper subset of the $\{q_i\}_{i=1}^{m}$. We show that this construction is compatible with families, i.e. given a family of pointed normalizations $(\pi:\tilde{\C} \rightarrow T, \{\sigma_i\}_{i=1}^{m})$, we consider
\begin{align*}
E:&=\oplus_{i=1}^{m}\sigma_i^*\O_{\tilde{\C}}(-\sigma_i),\\
\P:&=\P(E) \rightarrow T,
\end{align*}
and we construct a family of curves over $\P$, whose fibers range over all isomorphism classes of elliptic $m$-fold pointed curves whose normalization is obtained from a fiber $(\tilde{C}_t, \{\sigma_i(t)\}_{i=1}^{m})$  by sprouting semistable $\P^{1}$'s along a proper subset of $\{\sigma_i(t)\}_{i=1}^{m}$. In Section \ref{S:Stratification}, we use this construction to describe the strata $\E_{l}$ explicitly as projective bundles over products of moduli spaces of genus zero stable curves.

 In Section \ref{S:IntersectionTheory}, we establish a framework for doing intersection theory on $\M_{1,n}(m)$. The fact that $\M_{1,n}(m)$ may be non-normal for large $m$ presents a technical difficulty, which we circumvent by simply passing to the normalization $\M_{1,n}(m)^*$.  In Section \ref{S:PicardGroup}, we show that $\M_{1,n}(m)^*$ is $\Q$-factorial and that $\Pic_{\Q}(\M_{1,n}(m)^*)$ is naturally generated by tautological classes. In Section \ref{S:1ParameterFormulas}, we explain how to evaluate the degrees of tautological classes on one-parameter families of $m$-stable curves. The usual heuristics for nodal curves are not sufficient since families of $m$-stables curves exhibit novel features not encountered with stable curves. For example, one can have non-isotrivial families of $m$-stable whose pointed normalization \emph{is} isotrivial. Furthermore, whereas the limit of a node is always a node in a family of stable curves, non-disconnecting nodes degenerate to more complicated singularities in families of $m$-stable curves. We explain techniques for computing the degree of tautological classes on such families.
 
In Section \ref{S:MainTheorem}, we prove our main result. In Section \ref{S:Discrepancy}, we analyze the birational contraction $\phi: \M_{1,n} \dashrightarrow \M_{1,n}(m)^*$, and show that 
 $$\phi^*\phi_*D(s)-D(s) \geq 0 \text{ for } s \in (11-m,12-m).$$
This implies that the section ring of $D(s)$ on $\M_{1,n}$ is identical to the section ring of $\phi_*D(s)$ in $\M_{1,n}(m)^*$. Thus, to prove $\M_{1,n}^s=\M_{1,n}(m)^*$, it suffices to show that $\phi_*D(s)$ is ample. In Section \ref{S:AmpleDivisors}, we use the intersection theory developed in Section 3 to prove that $\phi_*D(s)$ has positive intersection on every curve in $\M_{1,n}(m)^*$ for $s \in (m,m+1)$. We then apply Kleiman's criterion to conclude that the divisor $D(s)$ is ample. Section \ref{S:Singularities} is logically independent of the rest of the paper; we use a discrepancy calculation to prove that the stacks $\S_{1,n}(m)$ must be singular for $m \geq 6$.

\textbf{Acknowledgements.} I am very grateful to Dawei Chen, Maksym Fedorchuk, Fred van der Wyck, Joe Harris, and Brendan Hassett for the numerous comments and ideas they offered throughout this project. During the preparation of this work, the author was partially supported by NSF grant 0901095.

\section{Geometry of $\S_{1,n}(m)$}\label{S:Geometry}

\subsection{Deformation theory}\label{S:DeformationTheory}

The deformation theory of stable curves implies that $\S_{1,n}$ is a smooth Deligne-Mumford stack with normal crossing boundary, and that $\S_{1,n}$ has a locally closed stratification by topological type. In this section, we investigate the corresponding properties for $\S_{1,n}(m)$. We assume the reader is familiar with formal deformation theory (as in \cite{Sernesi}), and consider the following deformation functors, from the category of Artinian $k$-algebras with residue field $k$ to sets.
\begin{align*}
\Def_{(C,\pn)}: &\,\, A \rightarrow \{\text{ Flat Deformations of $C$ over $A$ with $n$ sections $\sigma_1, \ldots, \sigma_n$ }\}\\
\Def_{C}: &\,\, A \rightarrow \{\text{ Flat Deformations of $C$ over $A$ }\} \\
\Def_{(q_i \in C)}: &\,\, A \rightarrow \{\text{ Flat deformations of $\Spec \O_{C,q_i}$ over $A$ }\}
\end{align*}

\begin{lemma}\label{L:FormalSmoothness} Suppose that $(C \pn)$ is a pointed curve with reduced singular points $q_1, \ldots, q_m \in C$. The natural morphisms of deformation functors
$$
\emph{\Def}_{(C,\pn)} \rightarrow \emph{\Def}_C \rightarrow \prod_{i=1}^{m}\emph{\Def}_{(q_i \in C)}
$$
are formally smooth of relative dimension $n$ and $h^1(C, \Omega_{C}^{\vee})$ respectively. 
\end{lemma}
\begin{proof}
Since the marked points $p_1, \ldots, p_n$ are smooth,
$$\Def_{(C,\pn)} \rightarrow \Def_{C}$$
is clearly formally smooth of relative dimension $n$. The fact that
$$\Def_{C} \rightarrow \prod_{i=1}^{m}\Def_{(q_i \in C)}$$
is formally smooth of relative dimension $h^1(C, \Omega_{C}^{\vee})$
is contained in \cite[Proposition 1.5]{DM} under the assumption that $C$ has local complete intersection singularities, but elliptic $m$-fold points are not local complete intersections for $m \geq 5$. Thus, we must use the cotangent complex. 

By \cite[C.4.8 and C.5.1]{GLS}, there exists a sequence of sheaves $\{\T^{i}_C: i \geq 0\}$, a sequence of finite-dimensional $k$-vector spaces $\{T^{i}_{C}: i \geq 0\}$, and a spectral sequence $E_{2}^{p,q}=H^p(\T^{q}_C) \rightarrow T^{p+q}_C$ with the following properties:
\begin{enumerate}
\item The sheaves $\{\T^{i}_{C}: i \geq 1\}$ are supported on the singular locus of $C$,
\item$\T^0_C=\Hom(\Omega_{C},\O_{C})$,
\item$T^{1}_C =\Def_{C}(k[\e]/(\e^2))$,
\item$T^{2}_C$ is an obstruction theory for $\Def_{C}$,
\item$H^0(C,(\T^{1}_{C})_{q}) =\Def_{(q \in C)}(k[\e]/(\e^2))$,
\item$H^0(C,(\T^{2}_{C})_{q})$ is an obstruction theory for $\Def_{(q \in C)}$.
\end{enumerate}

Since $C$ is a curve and $\T^{1}_C$ is supported on the singular locus, we have $H^2(\T^0_C)=0$ and $H^1(\T^1_C)=0$. The spectral sequence $E_2^{p,q}$ then gives an exact sequence
$$
0 \rightarrow H^1(\T_C^0) \rightarrow T_C^1 \rightarrow H^0(\T_C^1) \rightarrow 0 \rightarrow T^2_C \rightarrow H^0(\T_C^2).
$$
Since $\T_C^1$ and $\T_{C}^2$ are supported on the singular locus, we have
\begin{align*}
H^0(\T_C^1)&=\oplus_{i=1}^{m} H^0(X,(\T^{1}_{C})_{q_i})\\
H^0(\T_C^2)&=\oplus_{i=1}^{m} H^0(X,(\T^{2}_{C})_{q_i}).
\end{align*}
Thus, the exact sequence shows that $\Def_{C} \rightarrow \oplus_{i=1}^{m} \Def_{(q_i \in C)}$ induces a surjection on first-order deformations and an injection on obstruction spaces. Formal smoothness follows by \cite[Proposition 2.3.6]{Sernesi}. Finally, the relative dimension of the map on first-order deformations is evidently $\dim H^1(\T_C^0) = h^1(C, \Omega_{C}^{\vee})$.
\end{proof}

\begin{corollary}\label{C:Smoothness}
$\S_{1,n}(m)$ is smooth iff $m \leq 5$.
\end{corollary}
\begin{proof}
By Lemma \ref{L:FormalSmoothness}, $\S_{1,n}(m)$ is a smooth at a point $[C] \in \S_{1,n}(m)$ iff the local rings $\O_{C,p}$ have unobstructed deformations for all singular points $p \in C$. For $m=1,2,3$, the elliptic $m$-fold point is a local complete intersection, hence has unobstructed deformations. The cases $m=4,5$ are handled by slightly less well-known criteria: the local ring $\O_{C,p}$ is a Cohen-Macaulay quotient of a regular local ring of dimension three when $m=4$, and a Gorenstein quotient of a regular local ring of dimension four when $m=5$ \cite[Proposition 2.5]{SmythEI}. There is a determinental structure theorem for such local rings which implies that they have unobstructed deformations \cite[Theorem 8.3 and Theorem 9.7]{HarDef}. This shows that $\S_{1,n}(m)$ is smooth when $m \leq 5$. We will show that $\S_{1,n}(m)$ is singular for $m \geq 6$ in Section \ref{S:Singularities}.
\end{proof}

\begin{corollary}\label{C:Boundary}
The boundary $\Delta \subset \S_{1,n}(m)$ is normal crossing iff $m=0$.
\end{corollary}
\begin{proof}
If $m \geq 1$, then there exists an $m$-stable curve $(C, \pn)$ with a single cusp $q \in C$ and no other singular points. The family
$$
\Spec k[a,b,x,y]/(y^2=x^3+ax+b) \rightarrow \Spec k[a,b]
$$
is a miniversal deformation for the cusp and in these coordinates the locus of singular deformations is cut out by $b^2-4a^{3}$. It follows from Lemma \ref{L:FormalSmoothness} that, locally around $[C, \pn] \in \S_{1,n}(m)$, we can choose two smooth coordinate $a$ and $b$ such that $\Delta$ is defined by the equation $b^2-4a^3$. In particular, $\Delta$ is not a normal crossing divisor.
\end{proof}

\begin{corollary}[Stratification of $\S_{1,n}(m)$ by singularity type]\label{C:BasicStratification}
Consider the set-theoretic decomposition given by
$$
\S_{1,n}(m)= \mathcal{M}_{1,n} \coprod \E_{0} \coprod \E_{1} \coprod \ldots \coprod \E_{m},
$$
where $\E_{i}$
\begin{align*}
\E_{0}&:=\{[C] \in \S_{1,n}(m) | \text{ $C$ is singular with only nodal singularities}\},\\
\E_{l}&:=\{[C] \in \S_{1,n}(m) | \text{ $C$ has an elliptic $l$-fold point}\}.
\end{align*}
Then we have
\begin{enumerate}
\item $\E_{l} \subset \S_{1,n}(m)$ is a locally closed substack.
\item For $l \geq 1$, $\E_{l}$ is smooth.
\item$\E_{0}$ has normal crossing singularities and pure codimension one.
\item $\overline{\E_{l}} \subset \E_{l} \coprod \E_{l+1} \coprod \E_{l+2} \coprod \ldots \coprod \E_{m}.$
\end{enumerate}
\end{corollary}
\begin{proof} First, we show that if $l \geq 1$, then $\E_{l} \subset \S_{1,n}(m)$ is smooth and locally closed. Suppose $(C, \pn)$ is an $m$-stable curve with an elliptic $l$-fold point $q_0 \in C$ and nodes $q_1, \ldots, q_k \in C$. There exists an etale neighborhood of $[C, \pn]$, say
\begin{align*}
\pi: (U,0) &\rightarrow \S_{1,n}(m)\\
0 &\rightarrow [C,\pn],
\end{align*}
and a morphism
$$
s: U \rightarrow \prod_{i=0}^{k}\Ver(q_i \in C) \rightarrow \Ver(q_0 \in C)
$$
where $\Ver(q_i \in C)$ is the base of a miniversal deformation of the singularity $q_i \in C$. Note that $\pi^{-1}(\E_{l}) \subset U$ is simply the fiber of $s$ over $s(0) \in \Ver(C,q_0).$ Using Lemma \ref{L:FormalSmoothness} and the fact that the miniversal deformation space of a node is smooth, we conclude that $s$ is smooth in a neighborhood of $0$,  so $s^{-1}(s(0)) \subset U$ is a smooth, closed subvariety of $U$. It follows that $\E_{l} \subset \S_{1,n}(m)$ is smooth and locally closed.

The argument that $\E_{0}$ is locally closed with pure codimension one and normal crossing singularities is essentially identical: if $(C, \pn)$ is an $m$-stable curve with nodes $q_1, \ldots, q_k \in C$, there is an etale neigborhood $U$ of $[C,\pn] \in \S_{1,n}(m)$ and maps
$$
s_i: U \rightarrow \prod_{i=1}^{k} \Ver(q_i \in C) \rightarrow \Ver(C,q_i),
$$
and $\pi^{-1}(\E_{0})$ is the union of the fibers $s_{i}^{-1}(s_{i}(0))$ for $i=1, \ldots, k$.

Finally, to see that $\overline{\E_{l}}=\E_{l} \coprod \E_{l+1} \coprod \E_{l+2} \coprod \ldots \coprod \E_{m}$, it is sufficient to note that elliptic $m$-fold points only deform to elliptic $l$-fold points if $l < m$. This fact is proved in \cite[Lemma 3.10]{SmythEI}.
\end{proof}

In order to describe the strata $\E_{l}$ explicitly, we need to understand the moduli of attaching data of the elliptic $m$-fold point.

\subsection{Moduli of attaching data of the elliptic $m$-fold point}\label{S:AttachingModuli}
It is well-known that if $q \in C$ is node, then $C$ is determined (up to isomorphism) by its normalization $\tilde{C}$ and the two points $q_1, q_2$ lying above the node. Indeed, one can recover $C$ as follows: take $\tilde{C}/(q_1 \sim q_2)$ to be the underlying topological space of $C$ and define the sheaf of regular functions on $C$ to be the subsheaf of $\O_{\tilde{C}}$ generated by all functions which vanish at $q_1$ and $q_2$. By contrast, if $q \in C$ is an elliptic $m$-fold point, then the isomorphism class of $C$ is not determined by the pointed normalization $(\tilde{C}, \qm)$.

In order to study the moduli of attaching data of the elliptic $m$-fold point, let us fix a curve $\tilde{C}$ with $m$ distinct smooth points, say $q_1, \ldots, q_m \in C$, and define the following two sets
\begin{align*}
\Moduli &:=\{ (C,q)\,\,| \text{ $(C,q)$ satisfies $(a)$ and $(b)$} \}/\simeq,\\
\Maps&:=\{\pi:(\tilde{C}, \qm) \rightarrow (C,q)\,\,| \text{ $\pi$ satisfies $(a)$ and $(c)$} \}/\simeq,
\end{align*}
where the conditions $(a)$, $(b)$, and $(c)$ refer to
\begin{itemize}
\item[$(a)$] $q \in C$ is an elliptic $m$-fold point,
\item[$(b)$] The normalization of $(C,q)$ is isomorphic to $(\tilde{C},\qm)$,
\item[$(c)$] $\pi$ is the normalization of $(C,q)$.
\end{itemize}
As usual, an isomorphism between two maps, say $\pi:(\tilde{C}, \qm) \rightarrow (C,q)$ and $\pi':(\tilde{C}, \qm) \rightarrow (C',q')$, consists of an isomorphism $i:(C,q) \simeq (C',q')$ such that the obvious diagram commutes. There is a surjection
\begin{align*}
\Maps &\rightarrow \Moduli,\\
\intertext{given by forgetting the map, and two maps have the same image in moduli iff they differ by an automorphism of $(\tilde{C},\qm)$.
Thus, we have }
\Moduli &\simeq \Maps / \Aut(\tilde{C},\qm).\\
\end{align*}

\begin{remark}
For simplicity, we will assume that every automorphism of $\tilde{C}$ which fixes the set $\qm$ actually fixes the points $q_i$ individually. This holds when $(\tilde{C}, \qm)$ consists of $m$ distinct non-isomorphic connected components, each containing one of the points $q_i$, and this is the only case we need.
\end{remark}

Now let us consider the problem of parametrizing these sets algebraically. Given $$\pi:(\tilde{C},\qm) \rightarrow (C,q)$$ satisfying $(a)$ and $(c)$, Lemma \cite[Lemma 2.2]{SmythEI} implies that we obtain a codimension-one subspace
$$
\pi^*(T_{q}^{\vee}) \subset \oplus_{i=1}^{m} T_{q_i}^{\vee}
$$
satisfying $\pi^*(T_{q}^{\vee}) \supsetneq T_{q_i}^{\vee}$ for each $i=1, \ldots, m$. Since $\O_{C}$ can be recovered as the sheaf generated by (arbitrary lifts of) a basis of $\pi^*T_{q}^{\vee}$,  together with all functions vanishing to order at least two along $q_1, \ldots, q_m$, this subspace determines the map up to isomorphism. Conversely, any codimension-one subspace $$V\subset \oplus_{i=1}^{m} T_{q_i}^{\vee}$$
with the property that
$V \supsetneq T_{q_i}^{\vee}$ for any $i=1, \ldots, m$,
gives rise to a map $\pi: \tilde{C} \rightarrow C$ simply by identifying the points $q_1, \ldots, q_m$, and declaring $\O_{C}$ to be the push forward of the subsheaf of $\O_{\tilde{C}}$ generated by (arbitrary lifts of) a basis of $V$, together with all functions vanishing to order at least two along $q_1, \ldots, q_m$. By \cite[Lemma 2.2]{SmythEI}, the singular point $\pi(q_1)=\ldots =\pi(q_m) \in C$ is an elliptic $m$-fold point. In sum, we have established

\begin{lemma} Let $\P:= \P(\oplus_{j=1}^{m} T_{q_i}^{\vee})$ denote the projective space of hyperplanes in $ \oplus_{j=1}^{m} T_{q_i}^{\vee}$, and let $H_{i} \subset \P$ be the coordinate hyperplane $H_{i}:= \P(\oplus_{ j \neq i} T_{q_i}^{\vee})$
Then we have a natural bijection
\begin{align*}
\emph{\Maps} &\leftrightarrow \P \backslash (H_1 \cup \ldots \cup H_m)\\
\pi& \rightarrow \pi^*(T_{q}^{\vee})
\end{align*}
\end{lemma}
\begin{corollary}\label{C:mfoldmoduli}
If \emph{$\Aut(\tilde{C},\qm)=\{0\}$}, then we have a natural bijection
$$
\emph{\Moduli} \leftrightarrow \P \backslash (H_1 \cup \ldots \cup H_m)
$$
\end{corollary}
In the following lemma, we extend this description to the case when $(\tilde{C},\qm)$ has automorphisms.
\begin{lemma}\label{L:mfoldmoduli} Suppose that the image of the natural map
$$
\emph{\Aut}(\tilde{C},\qm) \rightarrow \oplus_{i=1}^{m}\emph{\Aut}(T_{q_i}^{\vee})
$$
is precisely
$$
\oplus_{i \in S}\emph{\Aut}(T_{q_i}^{\vee}),
$$
for some proper subset $S \subset \{1, \ldots, m\}$. Let $\P, H_1, \ldots, H_m$ be defined as before and set
$$H_{S}:=\cap_{i \in S}H_i=\P(\oplus_{i \notin S} T_{q_i}^{\vee}).$$
Then we have a natural bijection
$$
\emph{\Moduli} \leftrightarrow H_{S} \backslash \cup _{i \notin S} (H_i \cap H_{S}).
$$

\end{lemma}
\begin{proof}
Consider the map
\begin{align*}
\Maps \rightarrow \P \backslash (H_1 \cup \ldots \cup H_m) &\rightarrow H_{S} \backslash \cup _{i \notin S} (H_i \cap H_{S}),
\end{align*}
defined by
$$
\pi \rightarrow \pi^*(T_{q}^{\vee})\rightarrow \pi^*(T_{q}^{\vee}) \cap \oplus_{i \notin S} T_{q_i}^{\vee}.
$$
Two distinct maps differ by an element of $\Aut(\tilde{C},\qm)$ iff the corresponding subspaces $\pi^*(T_{q}^{\vee}) \subset \oplus_{i=1}^{m} T_{q_i}^{\vee}$ differ by an element of $\oplus_{i \in S}\Aut(T_{q_i}^{\vee})$. Since
$$
\Moduli \simeq \Maps / \Aut(\tilde{C},\qm),
$$
it suffices to show that two subspaces $\pi^*(T_{q}^{\vee}) \subset \oplus_{i=1}^{m} T_{q_i}^{\vee}$ differ by an element of $\oplus_{i \in S}\Aut(T_{q_i}^{\vee})$ iff they have the same projection $\pi^*(T_{q}^{\vee}) \cap \oplus_{i \notin S} T_{q_i}^{\vee}.$

To see this explicitly, order the branches so that $S=\{1, \ldots, k\}$, choose uniformizers $t_1, \ldots, t_m$ on the normalization, and pick coordinates for $\P \backslash (H_1 \cup \ldots \cup H_m)$ so that the point $(c_1, \ldots, c_m) \in (k^*)^m$ corresponds to the subspace spanned by
\[
\left(
\begin{matrix}
t_1& 0& \hdots  & 0 & c_1t_m\\
0&t_2& \ddots & \vdots & c_2 t_m\\
\vdots& \ddots& \ddots& 0& \vdots  \\
0 & \hdots &0 & t_{m-1} & c_{m-1}t_m
\end{matrix}
\right)
\]
The projection of this subspace to $\oplus_{i \notin S} T_{q_i}^{\vee}$ is simply
\[
\left(
\begin{matrix}
t_{k+1}& 0& \hdots  & 0 & c_{k+1}t_{k+1}\\
0&t_{k+2}& \ddots & \vdots & c_{k+2} t_{k+2}\\
\vdots& \ddots& \ddots& 0& \vdots  \\
0 & \hdots &0 & t_{m-1} & c_{m-1}t_m
\end{matrix}
\right)
\]
In these coordinates, an element $(\lambda_1, \ldots, \lambda_k) \in \oplus_{i \in S}\Aut(T_{q_i}^{\vee}) = (k^*)^{|S|}$ acts by
$$
(\lambda_1, \ldots, \lambda_k) * (c_1, \ldots, c_{m-1})=(\lambda_1^{-1}c_1, \ldots, \lambda_k^{-1}c_k, c_{k+1}, \ldots, c_{m-1}),
$$
which shows that two subspaces are in the same orbit iff they have the same projection to $\oplus_{i \notin S} T_{q_i}^{\vee}$.
\end{proof}

\begin{remark}
This entire discussion applies without change to the case of pointed curves, i.e. if we are given an $n$-pointed curve $(C,\pn)$ and $m$ smooth points $\qm \in C$ which are distinct from the marked points, we may define
\begin{align*}
\Moduli &:=\{ (C,q,\pn)\,\,| \text{ $(C,q,\pn)$ satisfies $(a)$ and $(b)$} \}/\simeq,\\
\Maps&:=\{\pi:(\tilde{C},\qm,\pn) \rightarrow (C,q,\pn)\,\,| \text{ $\pi$ satisfies $(a)$ and $(c)$} \}/\simeq,
\end{align*}
where the conditions $(a)$, $(b)$, and $(c)$ refer to
\begin{itemize}
\item[$(a)$] $p \in C$ is an elliptic $m$-fold point,
\item[$(b)$] The normalization of $(C,q, \pn)$ is isomorphic to $(\tilde{C},\qm,\pn)$,
\item[$(c)$] $\pi$ is the normalization of $(C,q,\pn)$.
\end{itemize}
Precisely the same arguments give
$$
\Moduli \simeq \Maps / \Aut(\tilde{C},\qm, \pn),
$$
and the statement and proof of Lemma \ref{L:mfoldmoduli} hold in this context, with $\Aut(\tilde{C},\qm)$ replaced by $\Aut(\tilde{C},\qm, \pn).$
\end{remark}
\subsection{Construction of universal elliptic $m$-fold pointed families}
If $(\tilde{C},\qm)$ is a fixed curve with $\Aut(\tilde{C}, \qm)=0,$ Corollary \ref{C:mfoldmoduli} implies
$$
\Moduli \simeq \P \backslash (H_1 \cup \ldots \cup H_m) \simeq (k^*)^{m-1}
$$
In this section, we construct a modular compactification $(k^*)^{m-1} \subset \P^{m-1}$ which is functorial with respect to the normalization $(\tilde{C},\qm)$. The key idea is to allow the normalization $(\tilde{C},\qm)$ to sprout a semistable $\P^{1}$ at $q_i$ as the modulus of attaching data approaches the hyperplane $H_i$ (see Figure \ref{F:compacttriplepoint}).

\begin{definition}[Sprouting]
Let $(\tilde{C},\qm)$ be an $m$-pointed curve, and $S \subset [m]$ a proper subset. We say that $(\tilde{C}',\qmp)$ is obtained from $(\tilde{C},\qm)$ by \emph{sprouting at $\{q_i\}_{i \in S}$} if 
$$
\tilde{C}' \simeq \tilde{C} \cup E_1 \cup \ldots \cup E_{|S|},
$$
where
\begin{enumerate}
\item $E_i$ is a smooth rational curve, nodally attached to $\tilde{C}$  at $q_i$,
\item For $i \in S$, $q_i'$ is an arbitrary point of $E_i - \{q_i\}$, 
\item For $i \notin S$, $q_i=q_i'$.
\end{enumerate}
Note that the isomorphism class of $(\tilde{C}',\qmp)$ is uniquely determined by $(\tilde{C},\qm)$ and the subset $S \subset [m]$.
\end{definition}
\begin{figure}
\scalebox{.40}{\includegraphics{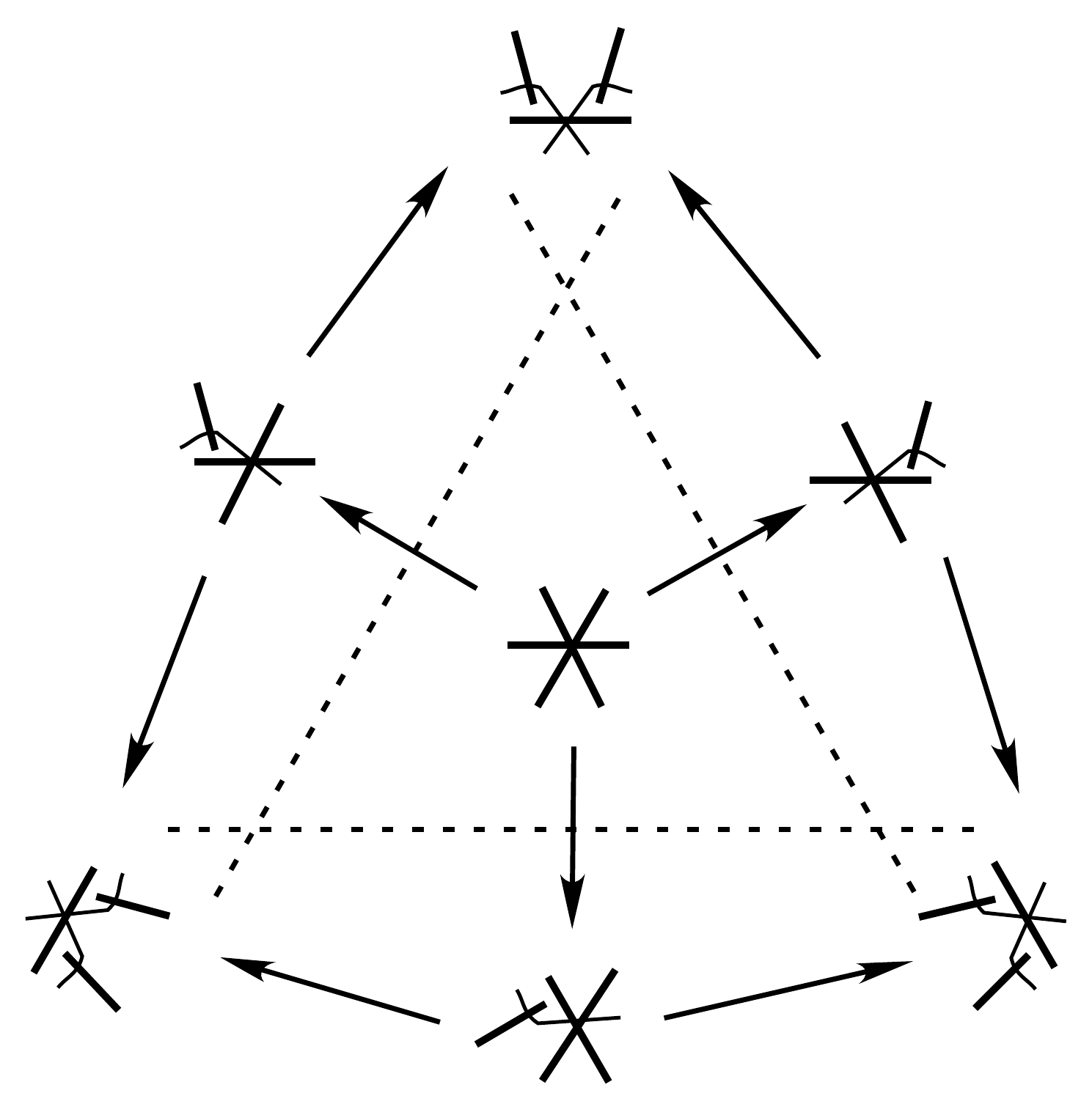}}
\caption{Compactification of the moduli of attaching data of the planar triple point. Over the three coordinate hyperplanes in $\P(T_{q_1}^{\vee} \oplus T_{q_2}^{\vee} \oplus T_{q_3}^{\vee})$, the normalization sprouts a $\P^1$ at the corresponding branch.}\label{F:compacttriplepoint}
\end{figure}

If $\Aut(\tilde{C},\qm)=0,$ and $(\tilde{C}',\qmp)$ is obtained from $(\tilde{C},\qm)$ by sprouting at $S$, then
$$
\Image \left( \Aut(\tilde{C}',\qmp) \rightarrow \oplus_{i=1}^{m}\Aut(T_{q'_i}^{\vee}) \right)=\oplus_{i \in S}\Aut(T_{q'_i}^{\vee}) .
$$
Thus, Lemma \ref{L:mfoldmoduli} implies that the attaching moduli for $(\tilde{C}',\qmp)$ is given by $H_{S} \backslash \cup _{i \notin S} (H_i \cap H_{S})$. As $S$ ranges over proper subsets of $[m]$, the locally closed subvarieties $H_{S} \backslash \cup _{i \notin S} (H_i \cap H_{S})$ give a stratification of $\P$. This suggests the construction of a flat family over $\P$ whose fibers
range over all isomorphism classes of elliptic $m$-fold pointed curves with pointed normalization obtained from $(\tilde{C},\qm)$ by sprouting along a proper subset of $\qm$. In fact, we can make this construction relative to a family of varying normalizations.

To set notation,  let $(f:\C \rightarrow T, \{\tau_i\}_{i=1}^{m})$ be a family of curves with $\taum$ mutually disjoint sections in the smooth locus of $f$. Let $\psi_{i}:=\tau_i^*\O_{\C}(-\tau_i)$ be the universal cotangent bundle along $\tau_i$, and consider the projective bundle
$$p: \P:=\P(\oplus_{i=1}^{m}\psi_i) \rightarrow T.$$
We will abuse notation by letting $f$ and $\tau_i$ continue to denote the pull-backs $p^*f$ and $p^*\tau_i$.
For any subset $S \subset [m]$, let $H_{S}$ denote the $\P^{m-|S|-1}$-subbundle of $\P$ corresponding to the quotient
$$
\oplus_{i=1}^{m}\psi_i  \rightarrow \oplus_{i \notin \{S\}}\psi_i \rightarrow 0,
$$
and set
$$
U_{S}:=H_{S} \backslash \cup_{i \notin \{S\}}(H_{i} \cap H_{S}).
$$
Note that, as $S$ ranges over non-empty subsets of $[m]$, the locally closed subschemes $U_{S}$ give a stratification of $\P$.

\begin{proposition}[Construction of universal elliptic $m$-fold pointed families I]\label{P:mfoldfamiliesI}
With notation as above, there exists a diagram
\[
\xymatrix{ 
&\tilde{\D} \ar[rd]^{\pi} \ar[ld]_{\phi} \ar[dd]^{\tilde{g}}&\\
\C \times_T \P(\oplus_{i=1}^{m}\psi_i)  \ar[rd]^{f}&& \D \ar[ld]_{g}\\
&\P(\oplus_{i=1}^{m}\psi_i)  \ar@/^1pc/[lu]^{\taum}  \ar@/^1pc/[uu]^{\{\tilde{\tau}_i\}_{i=1}^{m}} \ar@/_1pc/[ru]_{\tau}  &
}
\]
satisfying
\begin{itemize}
\item[(1)] $g, \tilde{g}$ are flat of relative dimension one.
\item[(2)] $\phi$ is the blow-up of $\C \times_T \P$ along the smooth codimension-two locus $\cup_{i=1}^{m}(\tau_i(\P) \cap f^{-1}(H_i))$, and $\tilde{\tau_i}$ is the strict transform of $\tau_i$. 
\item[(3)] $\pi$ is an isomorphism away from $\cup_{i=1}^{m}\tilde{\tau}_i$ and $\pi(\tilde{\tau}_1)=\ldots=\pi(\tilde{\tau}_m)=\tau$, i.e. $\pi$ is the normalization of $\D$ along $\tau$.
\item[(4)] For each geometric point $z \in  \P,$ $\tau(z) \in D_z$ is an elliptic $m$-fold point.\\
\end{itemize}
Furthermore, we can describe the restriction of this diagram to a geometric point $z \in \P$ as follows: Let $S \subset [m]$ be the unique proper subset (possibly empty) such that $z \in U_{S}$. Then
\begin{itemize}
\item[(5)] $\phi_{z}:(\tilde{D}_{z}, \{\tilde{\tau}_i(z)\}_{i=1}^{m}) \rightarrow (C_z, \{\tau_i(z)\}_{i=1}^{m})$, is the sprouting of $C_{z}$ along $\{\tau_{i}(z)\}_{i \in S}$. In particular, there is a canonical identification
$$
\oplus_{i \notin S}T_{C_{z},\tau_i(z)}^{\vee} = \oplus_{i \notin S}T_{\tilde{D}_z,\tilde{\tau}_i(z)}^{\vee}. 
$$
\item[(6)] $\pi_{z}:(\tilde{D}_{z}, \{\tilde{\tau}_i(z)\}_{i=1}^{m}) \rightarrow (D_z, \tau(z))$ is the normalization of $D_{z}$ at the elliptic $m$-fold point $\tau(z)$. The codimension-one subspace $\pi^*(T_{D_z,\tau(z)}^{\vee}) \subset \oplus_{i=1}^{m}T_{\tilde{D}_z,\tilde{\tau}_i(z)}^{\vee}$ satisfies
$$
\pi^*(T_{D_{z},\tau(z)}^{\vee}) \cap \oplus_{i \notin S} T_{\tilde{D}_z,\tilde{\tau}_i(z)}^{\vee} = [z] \cap \oplus_{i \notin S} T_{C_{z},\tau_i(z)}^{\vee}, 
$$
where $[z] \subset \oplus_{i=1}^{m}T_{C_{z},\tau_i(z)}^{\vee}$ is the codimension-one subspace  corresponding to $z \in \P$, and we identify $\oplus_{i \notin S}T_{C_{z},\tau_i(z)}^{\vee} = \oplus_{i \notin S}T_{\tilde{D}_z,\tilde{\tau}_i(z)}^{\vee}$ as in (5).
\end{itemize}
\end{proposition}

\begin{proof}
To construct the diagram, first note that for each $i=1, \ldots, m$, the codimension-two subvariety
$$f^{-1}(H_i) \cap \tau_i(\P) \subset \C \times_{T} \P$$
is contained in the smooth locus of $f$. Furthermore, these subvarieties are mutually disjoint. Let
$$
\phi: \tilde{\D} \rightarrow \C \times_{T} \P
$$
be the blow-up along the union of these subvarieties, let $E_1, \ldots, E_m$ denote the exceptional divisors of the blow-up, and let $\tilde{\tau}_i$ denote the strict transform of $\tau_i$. The flatness of $\tilde{g}: \tilde{\D} \rightarrow \P$ is a standard local calculation. Note that if $z \in U_{S}$, then the fiber over $z$ intersects the center of the blow-up transversely at $\tau_i(z)$ for $i \in S$, so property (5) is clear.

It remains to construct the map $\pi.$ Begin by considering the tautological sequence on $\P$:
$$\oplus_{i=1}^{m}p^*\psi_i \rightarrow \O_{\P}(1) \rightarrow 0,$$
and let $e_j \in \hom(p^*\psi_j,\O_{\P}(1))$ be the section obtained by the composition
$$e_j: p^*\psi_j \hookrightarrow \oplus_{i=1}^{m}p^*\psi_i \rightarrow \O_{\P}(1).$$
Note that $e_j$ vanishes to order one along $H_j$ and is non-vanishing elsewhere. Set 
$$\tilde{\psi}_i:=\tilde{\tau}_i^*\O_{\C}(-\tilde{\tau}_i),$$
and note that $\phi^* \tau_i = \tilde{\tau}_i+E_i$ implies
$$\tilde{\psi_i} =(p^*\psi_i)(H_i).$$
Since $e_i: p^*\psi_i \rightarrow \O_{\P}(1)$ vanishes to order one along $H_i$ and is non-vanishing elsewhere, $e_i$ induces an isomorphism
$$
\tilde{e}_i: \tilde{\psi}_i \simeq \O_{\P}(1).
$$
Taking the direct sum of these maps, we obtain an exact sequence
$$
0 \rightarrow \E \rightarrow \oplus_{i=1}^{m}\tilde{\psi}_i \rightarrow \O_{\P}(1) \rightarrow 0,
$$
with the property that, for each point $z \in \P$, the induced subspace
$$\E_{z} \subset \oplus_{i=1}^{m} T^{\vee}_{\tilde{\tau}_i(z)}.$$
does not contain any of the lines $T_{\tilde{\tau}_i(z)}^{\vee}$.

It is sufficient to define $\phi$ locally around $\tilde{\tau}_1, \ldots, \tilde{\tau}_m$,  so we may assume that $\tilde{g}$ is smooth and affine, i.e. we may assume
$$
\tilde{\D}:=\underline{\Spec}_{\O_{\P}} \tilde{g}_*\O_{\tilde{\D}}
$$
We specify a sheaf of $\O_{\P}$-subalgebras of $\tilde{g}_*\O_{\tilde{\D}}$ as follows: We consider the exact sequence on $\P$
$$
0 \rightarrow \tilde{g}_*\O_{\tilde{\D}}(-2\tilde{\tau}_1-\ldots-2\tilde{\tau}_m) \rightarrow \tilde{g}_*\O_{\tilde{\D}}(-\tilde{\tau}_1-\ldots-\tilde{\tau}_m) \rightarrow \oplus_{i=1}^{m}\tilde{\psi}_i \rightarrow 0,
$$
and let $\F \subset \tilde{g}_*\O_{\tilde{\D}}(-\tilde{\tau}_1-\ldots-\tilde{\tau}_m)$ be the inverse image of $\E \subset  \oplus_{i=1}^{m}\tilde{\psi}_i.$ Then we define $\G \subset \tilde{g}_*\O_{\tilde{\D}}$ to be the sheaf of $\O_{\P}$-subalgebras generated by sections of $\F$. Setting $\D:=\Spec_{\O_{\P}}\G$, we let $\pi$ be the morphism $\tilde{\D} \rightarrow \D$ associated to the inclusion $\G \subset \tilde{g}_*\O_{\tilde{\D}}$. 

Conclusion (3) is clear by construction, since any section of $\G$ vanishes along one section $\tau_i$ if and only if it vanishes along all of them. For (4), note that for any geometric point $z \in \P$, 
\begin{align*}
\pi_{z}^*\O_{D_z}(-2\tau(z))&=\O_{\tilde{D}_{z}}(-2\tilde{\tau}_1(z)-\ldots-2\tilde{\tau}_m(z)),\\
\pi_z^*(T_{\tau(z)}^{\vee})&=\E_{z} \subset \oplus_{i=1}^{m} T^{\vee}_{\tilde{\tau}_i(z)}.
\end{align*}
Since $\E_{z}$ does not contain any of the lines $T_{\tilde{\tau}_i(z)}^{\vee}$, $\tau(z) \in C$ is an elliptic $m$-fold point by \cite[Lemma 2.2]{SmythEI}. Finally, for (6), note that if $z \in U_{S}$, then the inclusion $p^*\psi_i  \subset \tilde{\psi}_i$ is an isomorphism in a neighborhood of $z$, for all $i \notin S$.  Thus, we have a commutative diagram
\[
\xymatrix{
\oplus_{i \notin S} \tilde{\psi}_i \ar[r]&\O_{\P}(1) \ar[r]& 0\\
\oplus_{i \notin S} p^*\psi_i \ar[u]^{\simeq} \ar[r]& \O_{\P}(1) \ar[r] \ar[u]^{\simeq} &0.
}
\]
The bottom arrow is induced by the tautological sequence, while the kernel of the top arrow is $\E \cap \oplus_{i \notin S} \tilde{\psi}_i$. It follows that
$$
\E_{z} \cap \oplus_{i \notin S}\oplus_{i=1}^{m} T^{\vee}_{\tilde{D}_{z}, \tilde{\tau}_i(z)}=[z] \cap \oplus_{i=1}^{m} T^{\vee}_{C_z,\tau_i(z)}.
$$
\end{proof}

\begin{corollary}\label{C:mfoldfamiliesI}
Suppose $(f: \C \rightarrow T, \taum, \sigman)$ is a family of $(n+m)$-pointed curves satisfying
\begin{enumerate}
\item The geometric fibers of $f$ have no automorphisms (as pointed curves).
\item No two geometric fibers of of $f$ are isomorphic (as pointed curves).
\end{enumerate}
Then the construction of Proposition \ref{P:mfoldfamiliesI} gives a family $(g: \D \rightarrow \P, \tau, \sigman)$ with the property that there is a bijection
\begin{align*}
\{\text{$k$-points $z \in \P$}\} &\leftrightarrow \{(D,q, \pn) \text{ \emph{satisfying (a),(b)}}\}/\simeq\\
z & \leftrightarrow (D_z, \tau(z), \{\sigma_i(z)\}_{i=1}^{m})
\end{align*}
where the conditions \emph{(a)} and \emph{(b)} are
\begin{itemize}
\item[(a)] $q \in D$ is an elliptic $m$-fold point,
\item[(b)] If $(\tilde{D},\qm, \pn)$ denotes the normalization of $(D,q, \pn)$  at $q$, then there exists a geometric fiber of $f$, say $(C_t,  \{\tau_i(t)\}_{i=1}^{m}, \{\sigma_i(t)\}_{i=1}^{n})$, and a proper subset $S \subset [m]$, such that $(\tilde{D},\qm, \pn)$ is obtained from $(C_t, \{\tau_i(t)\}_{i=1}^{m}, \{\sigma_i(t)\}_{i=1}^{n})$ by sprouting along $\{\tau_i(t)\}_{i \in S}$.
\end{itemize}
\end{corollary}
\begin{proof}
Note that the morphisms $\phi$ and $\pi$ constructed in Proposition \ref{P:mfoldfamiliesI} are isomorphisms in a neighborhood of the sections $\sigman$, so they induce sections $\sigman$ on $\D \rightarrow \P$.

To check the stated bijection, fix a geometric point $t \in T$ and a proper subset $S \subset [m]$, and let $(\tilde{D}, \qm, \pn)$ be the curve obtained from the fiber $f^{-1}(t)$ by sprouting along $\{\tau_i(t)\}_{i \in S}$. Since the fiber $f^{-1}(t)$ has no automorphisms, the automorphism group of $(\tilde{D}, \qm, \pn)$ is $(k^*)^{|S|}$, and we have
$$
\Image \left( \Aut(\tilde{D},\qm, \pn) \rightarrow \oplus_{i=1}^{m}\Aut(T_{q_i}^{\vee}) \right)=\oplus_{i \in S}\Aut(T_{q_i}^{\vee}) .
$$
Now Lemma \ref{L:mfoldmoduli} and Conclusion (6) of Proposition \ref{P:mfoldfamiliesI} imply that the fibers of $g$ over $p^{-1}(t) \cap U_{S}$ precisely range over all isomorphism classes of elliptic $m$-fold pointed curves whose pointed normalization is isomorphic to $(\tilde{D}, \qm, \pn).$ Since the locally closed subsets $U_{S}$ stratify $\P$, the fibers of $g$ over $p^{-1}(t)$ range over all isomorphism classes of elliptic $m$-fold pointed curves whose normalization is obtained from the fiber $f^{-1}(t)$ by sprouting along an arbitrary proper subset of $\{\tau_i(t)\}_{i=1}^{m}$. The claim follows.
\end{proof}

In Section 2.4, we will need a slight modification of Proposition \ref{P:mfoldfamiliesI}. Suppose we are given a family $(\C \rightarrow T, \{\tau_{i}\}_{i=1}^{l})$ with only $l$ attaching sections, where $l<m$. In Proposition \ref{P:mfoldfamiliesII}, we construct a universal family of elliptic $m$-fold pointed curves whose normalizations are the disjoint union of $m-l$ smooth rational curves and a curve obtained from a fiber of $f$ by sprouting along a proper subset of $\{\tau_i(t)\}_{i=1}^{l}$. As before, we define $\psi_i:=\tau_i^*\O_{\C}(-\tau_i)$, $p: \P:=\P(\oplus_{i=1}^{l}\psi_i) \rightarrow T,$
and abuse notation by letting $f$ and $\tau_i$ denote the pull-backs $p^*f$ and $p^*\tau_i$. Furthermore, for each $i=l+1, \ldots, m$, we define
$$
(\R^{i} \rightarrow \P, \tilde{\tau}_i)
$$
to be the one-pointed  $\P^{1}$-bundle $\P(\O_{\P} \oplus \O_{\P}(1)) \rightarrow \P$ with section $\tilde{\tau}_i$ corresponding to the quotient $\O_{\P} \oplus \O_{\P}(1) \rightarrow \O_{\P}$.

\begin{proposition}[Construction of universal elliptic $m$-fold pointed families II]\label{P:mfoldfamiliesII}
With notation as above, there exists a diagram
\[
\xymatrix{ 
\tilde{\D}^0 \ar_{\phi}[d] \ar^<<<<<<<<<<<<<<<<<<<<<<<<<<<{i}[rr]&&\tilde{\D}:=\tilde{\D}^0 \coprod \R^{l+1} \coprod \ldots \coprod \R^{m} \ar[d]^{\pi} \ar^<<<<<<<<<<<<<<{\tilde{g}}[ddl]\\
\C \times_T \P(\oplus_{i=1}^{m}\psi_i)  \ar[rd]^{f}&& \D \ar[ld]_{g}\\
&\P(\oplus_{i=1}^{l}\psi_i)  \ar@/^1pc/[lu]^{\taul}  \ar@/_1pc/[uul]_{\{\tilde{\tau}_i\}_{i=1}^{l}} \ar@/^1pc/[uur]^{\{\tilde{\tau}_i\}_{i=1}^{l}} \ar@/_1pc/[ru]_{\tau}  &
}
\]
satisfying
\begin{itemize}
\item[(1)] $g, \tilde{g}$ are flat of relative dimension one.
\item[(2)] $\phi$ is the blow-up of $\C \times_T \P$ along the smooth codimension-two locus $\cup_{i=1}^{l}(\tau_i(\P) \cap f^{-1}(H_i))$, and $\tilde{\tau_i}$ is the strict transform of $\tau_i$ for $i=1, \ldots, l$. 
\item[(3)]$i: \tilde{\D}^0 \rightarrow \tilde{\D}$ is the inclusion of $\tilde{\D}^0$ into the disjoint union $\tilde{\D}^0 \coprod \R^{l+1} \coprod \ldots \coprod \R^{m}$. 
\item[(4)] $\pi$ is an isomorphism away from $\cup_{i=1}^{m}\tilde{\tau}_i$ and $\pi(\tilde{\tau}_1)=\ldots=\pi(\tilde{\tau}_m)=\tau$.
\item[(5)] For each geometric point $z \in  \P,$ $\tau(z) \in D_z$ is an elliptic $m$-fold point. \\
\end{itemize}
Furthermore, we can describe the restriction of this diagram to a geometric point $z \in \P$ as follows. For any subset $S \subset [l]$, let $H_{S} \subset \P$ and $U_{S} \subset \P$ be defined as in Proposition \ref{P:mfoldfamiliesI}, and let $S \subset [l]$ be the unique subset such that $z \in U_{S}$. Then we have
\begin{itemize}
\item[(6)] $\phi_{z}:(\tilde{D}_{z}^0, \{\tilde{\tau}_i(z)\}_{i=1}^{l}) \rightarrow (C_z, \{\tau_i(z)\}_{i=1}^{l})$ is the sprouting of $C_{z}$ along $\{\tau_{i}(z)\}_{i \in S}$. In particular, there is a canonical identification
$$
\oplus_{i \in [l] \backslash S}T_{C_{z},\tau_i(z)}^{\vee} =\oplus_{i \in [l] \backslash S}T_{\tilde{D}^0_z,\tilde{\tau}_i(z)}^{\vee} =  \oplus_{i \in [l] \backslash S}T_{\tilde{D}_z,\tilde{\tau}_i(z)}^{\vee}. 
$$
\item[(7)] $\pi_{z}:(\tilde{D}_{z}, \{\tilde{\tau}_i(z)\}_{i=1}^{m}) \rightarrow (D_z, \tau(z))$ is the normalization of $D_{z}$ at the elliptic $m$-fold point $\tau(z)$. The codimension-one subspace $\pi^*(T_{D_z,\tau(z)}^{\vee}) \subset \oplus_{i=1}^{m}T_{\tilde{D}_z,\tilde{\tau}_i(z)}^{\vee}$ satisfies
$$
\phi^*(T_{D_{z},\tau(z)}^{\vee}) \cap \oplus_{i \in [l] \backslash S} T_{\tilde{D}_z,\tilde{\tau}_i(z)}^{\vee} = [z] \cap \oplus_{i \in [l] \backslash S} T_{C_{z},\tau_i(z)}^{\vee}, 
$$
where $[z] \subset \oplus_{i=1}^{l}T_{C_{z},\tau_i(z)}^{\vee}$ is the codimension-one subspace  corresponding to $z \in \P$, and we identify $\oplus_{i \in [l] \backslash S}T_{C_{z},\tau_i(z)}^{\vee} = \oplus_{i \in [l] \backslash S}T_{\tilde{D}_z,\tilde{\tau}_i(z)}^{\vee}$ as in (5).
\end{itemize}
\end{proposition}

\begin{proof}
The blow-up $\phi$ is constructed as in Proposition \ref{P:mfoldfamiliesI}. To construct $\pi$, we use the sections $\{\tilde{\tau}_i\}_{i=1}^{l}$ on $\tilde{\D}^0$ and the sections $\tilde{\tau}_{l+1}, \ldots, \tilde{\tau}_m$ on  $\R^{l+1}, \ldots, \R^m$. Set
$
\tilde{\psi}_i:=\tilde{\tau}_i^*\O_{\tilde{\D}}(-\tau_i)
$
and observe that, for each $i=1, \ldots, m$, we have a natural isomorphism
$$e_{i}:\tilde{\psi_i} \simeq \O_{\P}(1).$$
For $i=1, \ldots, l$, the existence of $e_i$ follows as in the proof of Proposition \ref{P:mfoldfamiliesI}. For $i=l+1, \ldots, m$, this is a standard computation on the projective bundle $\R^i$.
Taking the direct sum of the isomorphisms $e_{i}$, we obtain an exact sequence
$$
0 \rightarrow \E \rightarrow \oplus_{i=1}^{m} \tilde{\psi}_{i} \rightarrow \O_{\P}(1) \rightarrow 0,
$$
where $\E$ has the property that for each point $z \in \P$ induced subspace
$\E_{z} \subset \oplus_{i=1}^{m} T^{\vee}_{\tilde{\sigma}_i(z)}$
does not contain any of the lines $T_{\tilde{\sigma}_i(z)}^{\vee}$. Using $\E$, we may construct
$
\phi: \tilde{\D} \rightarrow \D,
$
and verify Properties (4)-(7) precisely as in Proposition \ref{P:mfoldfamiliesI}.
\end{proof}

Since each of the projective bundles $\R^{i} \rightarrow \P$ is endowed with a distinguished section disjoint from the attaching section (namely, the section corresponding to the quotient $\O_{\P} \oplus \O_{\P}(1) \rightarrow \O_{\P}(1) \rightarrow 0$), we may use the previous proposition to construct universal families of $n$-pointed elliptic $m$-fold points curves from $(n-m+l)$-pointed families of normalizations.

\begin{corollary}\label{C:mfoldfamiliesII}
Suppose $(f: \C \rightarrow T, \taul, \sigmav{n-m+l})$ is a family of pointed curves satisfying
\begin{enumerate}
\item The geometric fibers of $f$ have no automorphisms (as pointed curves).
\item No two geometric fibers of of $f$ are isomorphic (as pointed curves).
\end{enumerate}
Then the construction of Proposition \ref{P:mfoldfamiliesI} gives rise to a family of $n$-pointed curves $(g: \D \rightarrow \P, \tau, \sigman)$ with the property that there is a bijection
\begin{align*}
\{\text{$k$-points $z \in \P$}\} &\leftrightarrow \{(D,q, \pn) \text{ \emph{satisfying (a),(b)}}\}/\simeq\\
z \in \P & \rightarrow (D_z, \tau(z), \{\sigma_i(z)\}_{i=1}^{m})
\end{align*}
where the conditions \emph{(a)} and \emph{(b)} are
\begin{itemize}
\item[(a)] $q \in D$ is an elliptic $m$-fold point,
\item[(b)] The normalization of $(D,q, \pn)$  at $q$ is a disjoint union
$$
(\tilde{D}^0, \ql, \pv{n-m+l})\coprod \left( \coprod_{i=l+1}^{m} (R_i, q_i, p_{n-m+i}) \right)
$$
where $(\tilde{D}^0, \ql, \pv{n-m+l})$ is obtained from a geometric fiber of $f$, say $(C_t,  \{\tau_i(t)\}_{i=1}^{l}, \{\sigma_i(t)\}_{i=1}^{n-m+l})$, by sprouting along $\{\tau_i(t)\}_{i \in S}$ for some $S \subset [l]$, and each $(R_i, q_i, p_{n-m+i}) \simeq (\P^1, 0, \infty)$.
\end{itemize}
\end{corollary}
\begin{proof}
The morphisms $\phi$ and $\pi$ are isomorphisms in a neighborhood of $\sigmav{n-m+l}$, so $\sigmav{n-m+l}$ induce sections on $g:\D \rightarrow \P$. For $i=l+1, \ldots, m$, we define $\sigma_{n-m+i}$ to be the section of $\R^{i} \rightarrow \P$ corresponding to the quotient $\O_{\P} \oplus \O_{\P}(1) \rightarrow \O_{\P}(1) \rightarrow 0$. Since this section is disjoint from the attaching section $\tilde{\tau}_{i}$, it induces a section of $\D \rightarrow \P$. All together, we obtain a family of $n$-pointed curves $(\D \rightarrow \P, \sigman)$. The proof of the stated bijection is essentially identical to the proof of Corollary \ref{C:mfoldfamiliesI}.
\end{proof}

\subsection{Stratification by singularity type}\label{S:Stratification} In Section 2.1, we defined a stratification of $\S_{1,n}(m)$ by singularity type:
$$
\S_{1,n}(m)= \mathcal{M}_{1,n} \coprod \E_{0} \coprod \E_{1} \coprod \ldots \coprod \E_{m}.
$$
In this section, we construct the strata $\E_{l}$ $(l \geq 1)$ explicitly. We will show that the irreducible components (equivalently, by Corollary \ref{C:BasicStratification} (2), the connected components) of $\E_{l}$ are indexed by partitions of $[n]$ into $l$ subsets, i.e. we have
$$
\E_{l}= \coprod_{\Sigma} \E_{\Sigma},
$$ 
where $\Sigma$ runs over $l$-partitions of $[n]$. To describe the curves parametrized by the irreducible component $\E_{\Sigma}$, we need the following definition.
\begin{definition}[Combinatorial type]
Let $(C, \pn)$ be an $m$-stable curve with an elliptic $l$-fold point $q \in C$. Then the normalization of $C$ at $q$ consists of $l$ distinct connected components, each of which carries at least one of the marked points $\pn$. We define the \emph{combinatorial type of }$(C,\pn)$ to be the partition $\{S_{1}, \ldots, S_{l}\}$ of $[n]$ induced by the connected components of $\tilde{C}$.
\end{definition}
Given a partition $\Sigma:=\{S_1, \ldots, S_l\}$ of $[n]$, we will construct a universal family for all $m$-stable curves of combinatorial type $\Sigma$. We must consider two cases:
\begin{CaseI} Each $S_{i}$ satisfies $|S_i| \geq 2$.
\end{CaseI}
Let
$
f_i: \C_{i} \rightarrow  \M_{0,|S_1|+1} \times \ldots \times \M_{0,|S_l|+1}
$
be the pull back of the universal curve over $\M_{0,|S_i|+1}$, and label the tautological sections of $f_i$ as $\{\sigma_{j}: j \in S_i\} \cup \{\tau_i\}.$ Now apply Proposition \ref{P:mfoldfamiliesI} with 
\begin{align*}
T:=&\M_{0,|S_1|+1} \times \ldots \times \M_{0,|S_k|+1}\\
f:= &\coprod_{i=1}^{l} \C_{i} \rightarrow \M_{0,|S_1|+1} \times \ldots \times \M_{0,|S_l|+1}\\
\tau_i:= &T \rightarrow \C_{i} \hookrightarrow \coprod \C_{i}. 
\end{align*}
By Corollary \ref{C:mfoldfamiliesI}, we obtain an $n$-pointed family of curves $$(g: \D \rightarrow \P,\{\sigma_i\}_{i=1}^{n})$$
over the projective bundle
$
\P:=\P(\oplus_{i=1}^{l}\psi_i) \rightarrow T,
$
such that the fibers of $g$ range over all isomorphism classes of elliptic $l$-fold pointed curves $(D,q, \pn)$ whose normalization $(\tilde{D}, \ql, \pn)$ is obtained from a fiber of $f$ by sprouting along a proper subset of the points $\{\tau_i(t)\}_{i=1}^{l}$. 

Since the normalization of any $m$-stable curve of combinatorial type $\Sigma$ at its unique elliptic $l$-fold point is obtained from a disjoint union of $l$ stable curves of genus zero by sprouting along a subset of attaching points, every $m$-stable curve of combinatorial type $\Sigma$ appears as a fiber of $\D$. On the other hand, some fibers of $\D \rightarrow \P$ may be fail to be $m$-stable (i.e. they may have elliptic $l$-bridges for some $l<k \leq m$). Since $m$-stability is an open condition however \cite[Lemma 3.10]{SmythEI}, there is a maximal Zariski open subset $\E_{\Sigma} \subset \P$ such that the fibers of $g$ over $\E_{\Sigma}$ are $m$-stable, and we obtain an $m$-stable curve
$$(g: \C \rightarrow \E_{\Sigma}, \sigman).$$
whose fibers comprise all $m$-stable curves of combinatorial type $\Sigma$.

\begin{CaseII}
One or more of $S_{i}$ satisfy $|S_i|=1$.
\end{CaseII}
Order the $S_{i}$ so that $|S_i| \geq 2$ for $i=1, \ldots, k$, and $|S_i|=1$ for $i=k+1, \ldots, l$. For $i=1, \ldots, k$, let
$
f_i: \C_{i} \rightarrow  \M_{0,|S_1|+1} \times \ldots \times \M_{0,|S_k|+1}
$
be the pull-back of the universal curve over $\M_{0,|S_i|+1}$, and label the tautological sections of $f_i$ as $\{\sigma_{j}: j \in S_i\} \cup \{\tau_i\}.$ Now apply Proposition \ref{P:mfoldfamiliesII} with 
\begin{align*}
T:=&\M_{0,|S_1|+1} \times \ldots \times \M_{0,|S_k|+1}\\
f:= &\coprod_{i=1}^{k} \C_{i} \rightarrow \M_{0,|S_1|+1} \times \ldots \times \M_{0,|S_k|+1}\\
\tau_i:= &T \rightarrow \C_{i} \hookrightarrow \coprod \C_{i}.\\
\end{align*}
By Corollary \ref{C:mfoldfamiliesII}, we obtain a family of $n$-pointed curves $(g: \D \rightarrow \P, \sigman)$ over the projective bundle
$
\P:=\P(\oplus_{i=1}^{k}\psi_i) \rightarrow T,
$
such that the fibers of $g$ range over all isomorphism classes of curves whose normalization is a disjoint union $l-k$ smooth one-pointed rational curves and a curve obtained from a fiber of $f$ by sprouting along a proper subset of the points $\{\tau_i(t)\}_{i=1}^{k}$. Note that we consider the $l-k$ sections lying on the one-pointed rational components as labeled by the elements in $S_{k+1}, \ldots, S_{l}$.

As in Case I, there is a maximal Zariski open subset $\E_{\Sigma} \subset \P$ such that the fibers of $g$ over $\E_{\Sigma}$ are $m$-stable, and we obtain an $m$-stable curve
$$(g: \C \rightarrow \E_{\Sigma}, \sigman),$$
whose fibers comprise all $m$-stable curves of combinatorial type $\Sigma$.\\

\begin{proposition}\label{P:BoundaryStratification}
The natural classifying map
$$
\coprod_{\Sigma} \E_{\Sigma} \rightarrow \E_{l} \subset \S_{1,n}(m)
$$
is an isomorphism. In particular, the varieties $\E_{\Sigma}$ are the irreducible components of $\E_{l}$.
\end{proposition}
\begin{proof}
Since every point of $\E_{l}$ is an $m$-stable curve whose combinatorial type is given by some $l$-partition of $[n]$, the natural map
$$\coprod_{\Sigma}\E_{\Sigma} \rightarrow \E_{l}$$
is bijective on $k$-points. Since $\cup_{\Sigma}\E_{\Sigma}$ is smooth by construction and $\E_{l}$ is smooth by Corollary \ref{C:BasicStratification} (2), the morphism $\coprod_{\Sigma}\E_{\Sigma} \rightarrow \E_{l}$ is smooth. Since we are working in characteristic zero, a smooth morphism which is bijective on $k$-points is an isomorphism.
\end{proof}

\begin{corollary}\label{C:DimStrata}
The boundary stratum $\E_{l} \subset \S_{1,n}(m)$ has pure codimension $l+1$.
\end{corollary}

\begin{proof}
If $\Sigma:=\{S_1, \ldots, S_l\}$ is any $l$-partition of $[n]$, ordered so that $|S_{i}| \geq 2$ for $i=1, \ldots, k$, and $|S_{k+1}|=\ldots=|S_l|=1$, then $ \E_{\Sigma} \subset \E_{l}$ is an open subset of a projective bundle $\P(\oplus_{i=1}^{k}\psi_{i}) \rightarrow \M_{0,|S_1|+1} \times \ldots \M_{0,|S_k|+1}$. The dimension of this projective bundle is
$$\sum_{i=1}^{k}(|S_i|-2)+(k-1).$$
Since $\sum_{i=1}^{k}|S_i|=n-l+k,$ this expression reduces to $n-l-1$, as desired.
\end{proof}

\section{Intersection theory on $\S_{1,n}(m)$}\label{S:IntersectionTheory}

\subsection{The Picard group of $\M_{1,n}(m)^*$}\label{S:PicardGroup}
In this section, we will define several tautological divisor classes on $\S_{1,n}$, $\S_{1,n}(m)$, and $\S_{1,n}(m)^*$ (equivalently, $\M_{1,n}$, $\M_{1,n}(m)$, and $\M_{1,n}(m)^*$), and use these to give a complete description of $\Pic_{\Q}(\M_{1,n}(m)^*)$.  

We begin by recalling the definition of the tautological divisor classes on $\S_{1,n}$. If $\pi: \C \rightarrow \S_{1,n}$ is the universal curve, with universal sections $\sigma_1, \ldots, \sigma_n$, we have line bundles $\lambda, \psi_1, \ldots, \psi_n, \psi \in \Pic(\S_{1,n})$ defined as:
\begin{align*}
\lambda&= \text{det } (\pi_* \omega_{\C/ \S_{1,n}}),\\
\psi_i&=\sigma_i^*(\omega_{\C/\S_{1,n}}),\\
\psi&=\otimes_{i=1}^{n}\psi_i.
\end{align*}
To define the boundary divisors  of $\S_{1,n}$, we adopt the following terminology: If $(C,\pn)$ is an $n$-pointed curve of arithmetic genus one and $S \subset [n]$ is any subset, we say that $q \in C$ is a \emph{node of type $S$} if the normalization of $C$ at $q$ consists of two connected components (necessarily of genus zero and one), and $\{p_i \, | \, i \in S\}$ is the set of marked points supported on the genus zero component. We say that a node $q \in C$ is \emph{non-disconnecting} if the normalization of $C$ at $q$ is connected. We then define
\begin{align*}
\Delta_{irr} &:=  \{[C] \in \S_{1,n} | \text{ $C$ has a non-disconnecting node}\} \subset \S_{1,n},\\
\Delta_{0,S} &:= \{[C] \in \S_{1,n} | \text{ $C$ has a node of type $S$}\} \subset \S_{1,n},\\
\Delta_{0}&:= \{[C] \in \S_{1,n} | \text{ $C$ has a disconnecting node}\} \subset \S_{1,n}.
\end{align*}
$\Delta_{irr}$ and $\Delta_{0,S}$ are closed, irreducible, codimension one substacks of $\S_{1,n}$ when $|S| \geq 2$, while $\Delta_{0}=\sum_{S \subset [n]} \Delta_{0,S}$. Thus, we obtain cycles
$$\Delta_{irr}, \Delta_{0,S}, \Delta_{0} \in \A^1(\S_{1,n}).$$
Since the deformation space of a node is regular, these substacks are Cartier, and we obtain line bundles
$$
\delta_{irr}, \delta_{0,S}, \delta_{0} \in \Pic(\S_{1,n}).
$$

Now let us define the analagous tautological divisor classes on $\S_{1,n}(m)$.
We define $\lambda, \psi_1, \ldots, \psi_n, \psi \in \Pic(\S_{1,n}(m))$ by precisely the same recipes as on $\S_{1,n}$. Similarly, we define reduced closed substacks of $\S_{1,n}(m)$:
\begin{align*}
\Delta_{irr}&:=\{[C] \in \S_{1,n}(m) | \text{ $C$ has a non-disconnecting node or non-nodal singularity}\},\\ 
\Delta_{0,S}&:=\{[C] \in \S_{1,n}(m) | \text{ $C$ has a node of type $S$}\},\\
\Delta_{0}&:=\{[C] \in \S_{1,n}(m) | \text{ $C$ has a disconnecting node}\}.
\end{align*}
Note that $\Delta_{0,S}$ is non-empty iff $2 \leq |S| \leq n-m$. (The condition that $|S| \leq n-m$ comes from the requirement that $(C,\pn)$ have no elliptic $m$-bridge.) 

With this notation, $\Delta_{irr}, \Delta_{0,S}, \Delta_{0} \in \SV{n}{m}$ are simply the birational images of the corresponding divisors on $\S_{1,n}$. In particular, they are irreducible and we obtain
$$\Delta_{irr}, \Delta_{0,S}, \Delta_0 \in \A^1(\S_{1,n}(m)).$$
As before, each substack $\Delta_{0,S}$ is Cartier, so we obtain line bundles
$$
\delta_{0,S}, \delta_{0} \in \Pic(\S_{1,n}(m)).
$$
On the other hand, $\Delta_{irr} \subset \S_{1,n}(m)$ is \emph{not} obviously Cartier, so we do not immediately obtain a line bundle $\delta_{irr} \in  \Pic(\S_{1,n}(m))$.

Finally, we we will abuse notation by using $\lambda, \psi_i, \psi, \Delta_{irr}, \Delta_{0,S}$
to denote the line bundles and cycles on $\M_{1,n}(m)$ and $\M_{1,n}(m)^*$ induced by the canonical isomorphisms \cite[Proposition 6.1]{Vistoli}
\begin{align*}
\Pic_{\Q}(\S_{1,n}(m)) &\simeq \Pic_{\Q}(\M_{1,n}(m)),& \Pic_{\Q}(\NS{n}{m}) &\simeq \Pic_{\Q}(\NM{n}{m}),\\
\A^1_{\Q}(\S_{1,n}(m)) &\simeq \A^1_{\Q}(\M_{1,n}(m)),& \A^1_{\Q}(\NS{n}{m}) &\simeq \A^1_{\Q}(\NM{n}{m}).
\end{align*}
Note that the normalization of the coarse moduli space of a Deligne-Mumford stack is canonically isomorphic to the coarse moduli space of the normalization of the stack, so there is no ambiguity in the definition of $\M_{1,n}(m)^*$.

It is well-known that the tautological classes generate $\Pic_{\Q}(\M_{1,n})$, and we have a complete description of the relations between them. In the following Proposition (and throughout this section), we will use the notation $
[n]_{i}^{j}:=\{ S \subset [n] \,|\, i \leq |S| \leq j \}.
$

\begin{proposition}[$\Q$-Picard group of $\M_{1,n}$]\label{P:StartingPicardGroup}
\begin{enumerate}
\item[]
\item $\Pic_{\Q}(\M_{1,n})$ is freely generated by $\lambda$ and the boundary divisors $\{\delta_{0,S}\}_{S \in [n]_{2}^{n}}$.
\item The following relations hold in $\Pic_{\Q}(\M_{1,n})$:
\begin{align*}
\delta_{irr}&=12\lambda\\
\psi_i&=\lambda+\sum_{i \in S \in [n]_2^{n}} \delta_{0,S}\\
\psi&=n\lambda+\sum_{S \in [n]_2^{n}}|S| \, \delta_{0,S}
\end{align*}
\end{enumerate}
\end{proposition}
\begin{proof}
See \cite[Theorem 2.2]{AC}.
\end{proof}

We would like an analogue of Proposition \ref{P:StartingPicardGroup} for $\Pic_{\Q}(\M_{1,n}(m))$. Unfortunately, we do not know whether $\M_{1,n}(m)$ is normal, and this presents a major  obstacle. On the other hand, a description of $\Pic_{\Q}(\NM{n}{m})$ follows easily from Proposition \ref{P:StartingPicardGroup}.
\begin{proposition}[$\Q$-Picard group of $\NM{n}{m}$]\label{P:PicardGroup}
\begin{enumerate}
\item[]
\item The cycle map
$\Pic_{\Q}(\NM{n}{m}) \rightarrow A^{1}_{\Q}(\NM{n}{m})$ is an isomorphism. In particular, $\NM{n}{m}$ is 
$\Q$-factorial.
\item $\Pic_{\Q}(\NM{n}{m})$ is freely generated by $\lambda$ and $\{\delta_{0,S}\}_{S \in [n]_{2}^{n-m}}$.
\item The following relations hold in $\Pic_{\Q}(\NM{n}{m})$:
\begin{align*}
\psi_i&=\lambda+\sum_{i \in S \in [n]_2^{n-m}} \delta_{0,S}\\
\psi&=n\lambda+\sum_{S \in [n]_2^{n-m}}|S| \, \delta_{0,S}
\end{align*}
\end{enumerate}
\end{proposition}
\begin{proof}
Let $U \subset \M_{1,n}$ be the open set parametrizing $m$-stable curves, and let $\phi: \M_{1,n} \dashrightarrow \M_{1,n}(m)$ be the natural map. Then $\phi|_{U}$ is an isomorphism, and $\phi(U) \subset \M_{1,n}(m)$ is precisely the locus of nodal curves in $\M_{1,n}(m)$. In particular, $\phi(U)$ is smooth and, by Corollary \ref{C:DimStrata}, the codimension of $\M_{1,n}(m) \backslash \phi(U)$ is two.  

Now let $V$ be the maximal open subset on which the birational map $\M_{1,n} \dashrightarrow \M_{1,n}(m)^*$ is regular. Since the complement of $V$ in $\M_{1,n}$ has codimension two, Proposition \ref{P:StartingPicardGroup} gives
$$
\A^{1}_{\Q}(V)  \simeq A^{1}_{\Q}(\M_{1,n}) \simeq \Q\{\Delta_{irr}, \Delta_{0,S}: S \in [n]_{2}^{n}\}.
$$
Evidently, $U \subset V$ and the codimension one points of $V \backslash U$ are precisely the generic points of the divisors $\{\Delta_{0,S}: S \in [n]_{n-m+1}^{n}\}$. Thus, we have an exact sequence
$$\Q\{\Delta_{0,S}: S \in [n]_{n-m+1}^{n}\} \rightarrow \A^{1}_{\Q}(V)  \rightarrow \A^{1}_{\Q}(U) \rightarrow 0.$$
Since all boundary divisors are linearly independent in $\A^{1}_{\Q}(\M_{1,n})$, the map on the left is injective. Thus,
$$
\A^{1}_{\Q}(U) \simeq \Q\{\Delta_{irr}, \Delta_{0,S}: S \subset [n]_2^{n-m} \}. 
$$
Since $\phi|_{U}$ is an isomorphism, and the normalization map $\M_{1,n}(m)^* \rightarrow \M_{1,n}(m)$ is an isomorphism over $\phi(U)$, we have
$$
\A^{1}_{\Q}(\M_{1,n}(m)^*)  \simeq \A^{1}_{\Q}(\phi(U)) \simeq \Q\{\Delta_{irr}, \Delta_{0,S}: S \subset [n]_2^{n-m}\}.
$$
Now consider the map
$$
\Pic_{\Q}(\M_{1,n}(m)^*)  \rightarrow \A^{1}_{\Q}(\M_{1,n}(m)^*) .
$$
It is injective since $\M_{1,n}(m)^*$ is normal. To show that it is surjective, it suffices to see that $\delta_{0,S}$ maps to $\Delta_{0,S}$ and $12\lambda$ maps to $\Delta_{irr}$. This can be checked after restriction $\phi(U)$ since the complement has codimension two. But since $\phi|_{U}$ is an isomorphism, these follow from the corresponding statements on $\M_{1,n}$. Similarly, the stated relations can be checked after restriction to $\phi(U)$, where they follow from the corresponding relations in $\Pic_{\Q}(\M_{1,n})$.
\end{proof}

\begin{remark}
While we do not know whether $\Delta_{irr} \in A^1(\M_{1,n}(m))$ is $\Q$-Cartier, the proof of Proposition \ref{P:PicardGroup} shows that the cycle $\Delta_{irr} \in A^1(\NM{n}{m})$ \emph{is} $\Q$-Cartier with associated line bundle $12\lambda$.
\end{remark}

\subsection{Intersection theory on 1-parameter families}\label{S:1ParameterFormulas}
If $(f: \C \rightarrow B, \sigman)$ is an $n$-pointed $m$-stable curve over a smooth curve $B$, we obtain a classifying map
$$
c: B \rightarrow \S_{1,n}(m),
$$
and we wish to compute the intersection numbers
\begin{align*}
\lambda.B&:=\deg_{B}c^{*}\lambda,\\
\psi_i.B&:=\deg_{B}c^*\psi_i,\\
\delta_{0,S}.B&:=\deg_{B}c^{*}\delta_{0,S},
\end{align*}
in terms of the geometry of the family. Evidently, $\psi_i.B$ and $\delta_{0,S}.B$ may be computed by standard techniques: $\psi_i.B$ is $-\sigma_i^2$ and $\delta_{0,S}.B$ is the number of disconnecting nodes of type $S$ in the fibers of $f$, counted with multiplicity. Furthermore, since the limit of a node of type $S$ is a node of type $S$ in any family of $m$-stable curves, the case where $B \subset \Delta_{0,S}$ is handled in the usual way: normalizing $\C$ along the locus of nodes of type $S$ and letting $\tau_1$, $\tau_2$ be the sections lying over this locus, we have $\delta_{0,S}.B=\tau_1^2+\tau_2^2$.

In this section, we explain how to compute $\lambda.B$ for arbitrary 1-parameter families of $m$-stable curves. First, we consider the special case where the classifying map $c: B \rightarrow \S_{1,n}(m)$ factors through one of the equisingular boundary strata $\E_{l}$, i.e. when every fiber of $f$ has an elliptic $l$-fold point. In this case, we compute $\lambda.B$ as a certain self-intersection on the surface obtained by normalizing along the locus of elliptic $l$-fold points (Proposition \ref{P:LambdaNumber}). Then we use stable reduction to reduce the general case to this special case (Corollary \ref{C:CaseIIFormulas}).

\begin{CaseI}
$c: B \rightarrow \S_{1,n}(m)$ factors through a boundary stratum $\E_{l}$ ($l \geq 1$).
\end{CaseI}

Since $f: \C \rightarrow B$ has a unique elliptic $l$-fold point in each fiber, $f$ admits a section $\tau$ such that $\tau(b) \in C_{b}$ is an elliptic $l$-fold point for each $b \in B$. Let $\pi:\tilde{\C} \rightarrow \C$ be the normalization of $\C$ along $\tau$, and let $\tilde{\tau}_1, \ldots, \tilde{\tau}_l$ by the sections lying over $\tau$.
\begin{proposition}\label{P:LambdaNumber}
With notation as above, $\lambda.B=\tilde{\tau}_{i}^2$ for any $i \in \{1, \ldots, l\}$.
\end{proposition}
\begin{proof}
Consider the sheaf homomorphism $\tau^*\I_{\tau} \rightarrow \oplus_{i=1}^{l}\tilde{\tau}_i^*\I_{\tilde{\tau}_i}$, whose restriction to the fiber over $b \in B$ is just the map
$$
m_{\tau(b)}/m_{\tau(b)}^2 \rightarrow \oplus_{i=1}^{l}m_{\tilde{\tau}_i(b)}/m_{\tilde{\tau}_i(b)}^2.
$$
Since $\tau(b) \in C_b$ is an elliptic $l$-fold point, \cite[Lemma 2.2 (1)]{SmythEI} implies that this map has a 1-dimensional quotient. Since this holds on every fiber, we have an invertible quotient sheaf $\L$, defined by the exact sequence
$$
\tau^*\I_{\tau} \rightarrow \oplus_{i=1}^{l}\tilde{\tau}_i^*(\I_{\tilde{\tau}_i}) \rightarrow \L \rightarrow 0.
$$
\cite[Lemma 2.2 (2)]{SmythEI} implies that each composition
$$
\tau_{i}^*\I_{\tilde{\tau}_i} \hookrightarrow \oplus_{i=1}^{l}\tilde{\tau}_i^*\I_{\tilde{\tau}_i} \rightarrow \L 
$$
is nowhere vanishing. Since $\tilde{\tau}_{i}^*\I_{\tilde{\tau}_i}$  and $\L$ are invertible, these must be isomorphisms. Thus, we have
$$
\psi_i:= \tilde{\tau}_{i}^*\I_{\tilde{\tau}_i}^{\vee} \simeq \L^{\vee},
$$
for each $i \in \{1, \ldots, l\}$. Note that since $\tilde{f}: \tilde{\C} \rightarrow B$ is a family of genus zero curves, we have $c_1(\tilde{\pi}_*\omega_{\tilde{\C}/B})=0$. Thus, to prove the proposition, it suffices to show
\begin{align*}
c_1(\pi_*\omega_{\C/B})=c_1(\tilde{\pi}_*\omega_{\tilde{\C}/B})+c_1(\L^{\vee}).\\
\end{align*}

To prove this formula, we must recall some facts about the dualizing sheaf of an elliptic $m$-fold pointed curve. If $C$ is a complete curve with an elliptic $l$-fold point $q \in C$, $\pi: \tilde{C} \rightarrow C$ is the normalization of $C$ at $q$, and $q_1, \ldots, q_l \in \tilde{C}$ are the points lying above $q$, we may compare $\omega_{C}$ and $\omega_{\tilde{C}}$ as follows: For any section $\omega \in \omega_{\tilde{C}}(2q_1+ \ldots + 2q_l)$, let $(\omega): \oplus_{i=1}^{l} m_{q_i}/m_{q_i}^2 \rightarrow k,$ denote the linear functional induced by
\begin{align*}
f & \rightarrow \sum_{i=1}^{l}\Res_{q_i}(f\omega), \,\,\, f \in \oplus_{i=1}^{l}m_{q_i}.
\end{align*}
In \cite[Section 2.2]{SmythEI}, we showed that $\omega_{C} \subset \pi_*\omega_{\tilde{C}}(2q_1+ \ldots + 2q_l)$, is precisely the subsheaf of sections satisfying:
\begin{itemize}
\item[(a)] $\sum_{i=1}^{l}\Res_{q_i}\omega=0$, and
\item[(b)] $(\omega) \in \Ker(  \oplus_{i=1}^{l} (m_{q_i}/m_{q_i}^2)^{\vee} \rightarrow (m_{q}/m_{q}^2)^{\vee})$.
\end{itemize}

We make use of this observation by considering the following two-step filtration for $f_*\omega_{\C/B}$:
$$
\tilde{f}_*\omega_{\tilde{\C}/B} \subset f_*\omega_{\C/B} \cap \tilde{f}_*\omega_{\tilde{\C}/B}(\tilde{\tau}_1+\ldots+\tilde{\tau}_l)  \subset  f_*\omega_{\C/B} \cap \tilde{f}_*\omega_{\tilde{\C}/B}(2\tilde{\tau}_1+\ldots+2\tilde{\tau}_l)=f_*\omega_{\C/B}.
$$
Define $\Lambda$ and $\Lambda'$ to be the quotients of this filtration, i.e.
$$
0 \rightarrow \tilde{f}_*\omega_{\tilde{\C}/B} \rightarrow f_*\omega_{\C/B} \cap \tilde{f}_*\omega_{\tilde{\C}/B}(\tilde{\tau}_1+\ldots+\tilde{\tau}_l) \rightarrow \Lambda' \rightarrow 0,$$
$$
0  \rightarrow  f_*\omega_{\C/B} \cap \tilde{f}_*\omega_{\tilde{\C}/B}(\tilde{\tau}_1+\ldots+\tilde{\tau}_l) \rightarrow f_*\omega_{\C/B} \rightarrow  \Lambda \rightarrow 0.
$$
It suffices to show that $c_1(\Lambda')=0$ and $c_1(\Lambda)=c_1(\L^{\vee}).$ To check that $c_1(\Lambda')=0$, consider the sequence
$$
0 \rightarrow \tilde{f}_*\omega_{\tilde{\C}/B} \rightarrow \tilde{f}_*\omega_{\tilde{\C}/B}(\tilde{\tau}_1+\ldots +\tilde{\tau}_l) \rightarrow \oplus_{i=1}^{l}\O_{B},
$$
where we have used the canonical isomorphism $\omega_{\tilde{\C}/B}(\tilde{\tau}_{i})|_{\tilde{\tau}_{i}} \simeq \O_{\tilde{\tau}_i}$ coming from adjunction. Since the map $\tilde{f}_*\omega_{\tilde{\C}/B}(\tilde{\tau}_1+\ldots +\tilde{\tau}_l) \rightarrow \oplus_{i=1}^{l}\O_{B}$ is given by taking residues, condition (a) implies that $\Lambda'$ lies in an exact sequence
$$
0 \rightarrow \Lambda' \rightarrow \oplus_{i=1}^{l}\O_{B}  \rightarrow \O_{B} \rightarrow 0,
$$
where $\oplus_{i=1}^{l}\O_{B}  \rightarrow \O_{B}$ is given summing sections. Thus, $c_1(\Lambda')=0$.

To check $c_1(\Lambda)=c_1(\L^{\vee}),$ consider the sequence
$$
0 \rightarrow \tilde{f}_*\omega_{\tilde{\C}/B}(\tilde{\tau}_1+\ldots +\tilde{\tau}_l) \rightarrow \tilde{f}_*\omega_{\tilde{\C}/B}(2\tilde{\tau}_1+\ldots +2\tilde{\tau}_l) \rightarrow \oplus_{i=1}^{l}\tilde{\tau}_i^{*}\I_{\tilde{\tau}_i}^{\vee},
$$
where we have used the canonical isomorphism $\omega_{\tilde{\C}/B}(2\tilde{\tau}_{i})|_{\tilde{\tau}_{i}} \simeq \I_{\tilde{\tau}_i}^{\vee}|_{\tilde{\tau}_i}$ coming from adjunction. Now condition (b) implies that $\Lambda$ is simply the kernel of the map
$$
 \oplus_{i=1}^{l}\tilde{\tau}_i^*\I_{\tilde{\tau}_i}^{\vee} \rightarrow (\tau^*\I_{\tau})^{\vee} \rightarrow 0,
$$
i.e. $\Lambda \simeq \L^{\vee}$ as desired.
\end{proof}

\begin{example} 
Recall that the connected components of $\E_{l}$ are parametrized by partitions of $[n]$ (Proposition \ref{P:BoundaryStratification}). 
Given a partition $\Sigma=\{S_{1}, \ldots, S_{l}\}$, with $S_1, \ldots, S_k$ satisfying $|S_i| \geq 2$ and $|S_{k+1}|=\ldots=|S_{l}|=1$, the associated connected component $\E_{\Sigma}$ is simply the projective bundle
$$
\P(\psi_1 \oplus \ldots \oplus \psi_k) \rightarrow \M_{0,|S_1|+1} \times \ldots \times \M_{0,|S_k|+1}.
$$
Let $B=\P^{1}$ be a generic fiber of this projective bundle and let $(f: \C \rightarrow B, \sigman)$ be the associated family of $m$-stable curves. We will compute the intersection numbers $\psi_i.B$, $\delta_{0,S}.B$, $\lambda.B$ for this family.

Unwinding the construction of $\E_{\Sigma}$ in Proposition \ref{P:mfoldfamiliesII}, we find that $\C \rightarrow B$ can be explicitly described as follows. For each $i=1, \ldots, k$, choose a point $z_i \in \P^{1}$ and a smooth genus zero stable curve $(C_i, \{p_j\}_{j=1}^{|S_i|}, q_i)$, and let
\begin{align*}
&\tilde{\C_i}=\text{Blow-up of $C_i \times \P^1$ at $(q_i,z_i)$}.
\end{align*}
Let $\{\sigma_j\}_{j=1}^{|S_i|}$ and $\tau_i$ be the strict transforms of the sections $\{p_j\} \times \P^{1}$ and $\{q_i\} \times \P^{1}$ so that we obtain a family of $(|S_i|+1)$-pointed genus zero curves $(\C_i \rightarrow B, \{\sigma_j\}_{j \in S_i}, \tau_i)$.

In addition, for $i=k+1, \ldots, l$, let
\begin{align*}
&\C_i:=\P(\O_{\P_1} \oplus \O_{\P_1}(1)) \rightarrow B = \P^1
\end{align*}
and label a pair of disjoint sections with self-intersections $1$ and $-1$ by $\sigma_j$ $(j \in S_i)$, and $\tau_i$ respectively, so that we obtain a 
two-pointed family of genus zero curves $(\C_i \rightarrow B, \{\sigma_j\}_{j \in S_i}, \tau_i).$

The family $(\C \rightarrow B, \sigman)$ is constructed by gluing $\{\C_i\}_{i=1}^{l}$ along the sections $\tau_1, \ldots, \tau_l$. The gluing data corresponds to a one-dimensional quotient of the vector bundle $\oplus_{i=1}^{l}  \tau_i^*\O_{\C_i}(-\tau_i)$ which is constructed as follows. For each $i=1, \ldots, l$, we have an isomorphism $\tau_i^*\O_{\C_i}(-\tau_i) \simeq \O_{\P}(1)$, and taking the direct sum of these maps gives the quotient $\oplus_{i=1}^{l}  \tau_i^*\O_{\C_i}(-\tau_i) \rightarrow \O_{\P}(1) \rightarrow 0$.

Since the sections $\tau_1, \ldots, \tau_l$ lying above the locus of elliptic $l$-fold points each have self-intersection $-1$, Proposition \ref{P:LambdaNumber} implies that $\lambda.B=-1$. The remaining intersection numbers are apparent from the construction:

\begin{align*}
\lambda.B&=-1\\
\delta_{0,S}.B&=
\begin{cases}
1&\text{ if } S \in \{S_1, \ldots, S_k\}\\
0& \text{ otherwise }\\
\end{cases}\\
\psi_{i}.B&=
\begin{cases}
-1&\text{ if $i \in  \{S_{k+1}, \ldots, S_{l}\}$}\\
0 & \text{ otherwise.}\\
\end{cases}
\end{align*}

Note that neither $\lambda$ nor $\psi_i$ is nef on $\S_{1,n}(m)$ whereas both are nef on $\S_{1,n}$.
\end{example}

\begin{CaseII}
$c: B \rightarrow \S_{1,n}(m)$ does not factor through any boundary stratum $\E_{l}$.
\end{CaseII}

We reduce to Case I as follows: Suppose the generic fiber of $\C \rightarrow B$ contains an elliptic $l$-fold point. (If the generic fiber is smooth or nodal, take $l=0$.) Outside a finite set of fibers, $(f:\C \rightarrow B, \sigman)$ is $l$-stable, so (after a finite base change) there exists family of $l$-stable curves $(g:\D \rightarrow B, \{\sigma_i\}_{i=1}^{n})$, and a birational map $\D \dashrightarrow \C$ over $B$. We obtain a commutative diagram
\[
\xymatrix{
&B \ar[dl]_{c_l} \ar[dr]^{c_m}&\\
\S_{1,n}(l) \ar@{-->}[rr]&& \S_{1,n}(m)
}
\]
where $c_l$ is the classifying map associated to the $l$-stable family and $c_m$ is the classifying map associated to the $m$-stable family. We will use the notation
\begin{align*}
\lambda^{m}.B&:=\deg_Bc_m^*\lambda&\lambda^{l}.B&:=\deg_Bc_l^*\lambda\\
\psi^{m}.B&:=\deg_{B} c_{m}^{*}\psi&\psi^{l}.B&:=\deg_{B} c_{l}^{*}\psi\\
\delta_0^{m}.B&:=\deg_Bc_{m}^*\delta_0&\delta_0^{l}.B&:=\deg_Bc_{l}^*\delta_0.
\end{align*}
Since the boundary stratum $\E_{l} \subset \SV{n}{l}$ is closed, the image $c_l(B)$ lies entirely in $\E_{l}$ and we can compute $\lambda^{l}.B$ as in Case I. The intersection number we are after are $\lambda^{m}.B$, so we are left with the problem of computing the difference $\lambda^{m}.B-\lambda^{l}.B$. We will explain how to compute this difference in terms of the explicit sequence of blow-ups and contractions that transforms the fibers of $\D \rightarrow B$ into the fibers of $\C \rightarrow B$.
 
For simplicity, let us assume that the generic fiber of $\C$ has no disconnecting nodes, and that $\D$ and $\C$ are isomorphic away from the fiber over a single point $b \in B$.

\begin{claim} There exists a diagram
\[
\xymatrix{
&\B_{0} \ar[dl]_{p_0} \ar[dr]^{q_0}&&\B_{1}  \ar[dl]_{p_1} \ar[dr]^{q_1}&&&&\B_{k}  \ar[dl]_{p_k} \ar[dr]^{q_k}\\
\C_{0}&&\C_{1}&&\C_{2}& \cdots &\C_{k-1}&&\C_{k}
}
\]
satisfying
\begin{enumerate}
\item $\C_{0} \rightarrow \D$ is the desingularization of $\D$ at the disconnecting nodes of $D_b$.
\item $\C_{k} \rightarrow \C$ is the desingularization of $\C$ at the disconnecting nodes of $C_b$.
\item $p_i$ is the blow-up of $\C_{i}$ at a collection of smooth points of $\C_{i}$, namely the marked points of the minimal elliptic subcurve of $(\C_{i})_b$. 
\item $q_i$ is a birational contraction with $\Exc(q_i)=E_i$, where $E_i$ is the minimal elliptic subcurve of $(\B_{i})_b$. 
\end{enumerate}
(Recall that the \emph{minimal elliptic subcurve} of a Gorenstein genus one curve $C$ is the unique connected genus one subcurve $E \subset C$ such that $E$ has no disconnecting nodes \cite[Lemma 3.1]{SmythEI}.)
\end{claim}
\begin{proof}
This diagram is constructed precisely as in the proof of the valuative criterion for $\S_{1,n}(m)$ (see \cite[Theorem 3.11]{SmythEI} and \cite[Figure 5]{SmythEI}). For the convenience of the reader, we recall the argument. Given $\C_i$, we may certainly blow-up along the collection of marked points of the minimal elliptic subcurve of $(\C_i)_b$ to obtain $p_i$. To construct $q_i$, it suffices by \cite[Lemma 2.12]{SmythEI} to exhibit a nef line bundle on $\B_i$ which has degree zero precisely on the minimal elliptic subcurve $E_i \subset (\B_i)_b$. One easily checks that the line bundle $\omega_{\B_i/B}(E_i+2\Sigma_{i=1}^{n} \sigma_i)$ satisfies this condition.

It only remains to check that, after finitely-many steps, we arrive at the desingularization of the $m$-stable limit. To see this, one first checks (as in Step 2 of the proof of \cite[Theorem 3.11(1)]{SmythEI}) that the contraction $q_i$ replaces $E_i$ by an elliptic $l_i$-fold point where $l_i:=|E_i \cap E_i^{c}|$ and that number of disconnecting nodes in $\C_{i+1}$ is less than the number of disconnecting nodes in $\C_i$. This implies that after finitely many steps, we arrive at a special fiber of $\C_{i+1}$ which has no elliptic $j$-bridge ($j \leq m$) and has only nodes and elliptic $j$-fold points ($j \leq m$) as singularities. Letting $\C_{k} \rightarrow \C$ denote the morphism obtained by blowing down all semistable chains of $\P^{1}$'s, one checks (as in Step 3 of the proof of \cite[Theorem 3.11(1)]{SmythEI}) that the special fiber of $\C$ is $m$-stable. By uniqueness of $m$-stable limits, the resulting family of $m$-stable curves must be the family $(\C \rightarrow B, \sigman)$.
\end{proof}

Fixing a diagram as above, let $F_i$ be the minimal elliptic subcurve of the fiber $(\C_{i})_b$, and define\begin{align*}
n_i:=&|\{ \sigma_i | \sigma_i(b) \in F_i \}|,\\
m_i:=&|F_i \cap \overline{(\C_i)_b \backslash F_i}|,\\
l_i:=&n_i+m_i.
\end{align*}
We call $l_i$ the \emph{level} of the minimal elliptic subcurve $F_i \subset (\C_i)_b$. With this notation, we can record formulae not only for the difference $\lambda^{m}.B-\lambda^{l}.B$, but also for $\psi^{m}.B-\psi^{l}.B$ and $\delta_0^{m}.B-\delta_0^{l}.B$.
\begin{proposition}\label{P:IntersectionNumbers}
With notation as above, we have
\begin{align*}
\lambda^{m}.B-\lambda^{l}.B&=k\\
\psi^{m}.B-\psi^{l}.B&=\sum_{i=0}^{k-1}n_i\\
\delta_0^{m}.B-\delta_0^{l}.B&=-\sum_{i=0}^{k-1}m_i\\
(\psi^{m}-\delta_0^m).B-(\psi^{l}-\delta_0^l).B&=\sum_{i=1}^{k-1}l_i
\end{align*}
\end{proposition}

\begin{proof}
Let $g^i$ denote the structure morphism $g^i:\C_i \rightarrow B$, and $h^i$ the structure morphism $h^i:\B_i \rightarrow B$. For the first formula, we must show that
$$
c_1(f_*\omega_{\C/B})= c_1(g_*\omega_{\D/B})+k.
$$
Note that since the desingularization maps $\C_0\rightarrow \D$ and $\C_k \rightarrow \C$ are obtained by resolving $A_{k}$-singularities, we have
\begin{align*}
g_*\omega_{\D/B}&=g^0_*\omega_{\C_0/B}\\
 f_*\omega_{\C/B}&=g^k_*\omega_{\C_k/B}
\end{align*}
Thus, it is enough to show that for each $i=0, \ldots, k-1$,
$$
c_1 (g^{i+1}_*\omega_{\C^i/B})=c_1 (g^{i}_*\omega_{\C^{i-1}/B})+1.
$$
Let $R_1, \ldots, R_{n_i}$ be the exceptional divisors of the blow-up $p_i$ and let $E_i$ be the exceptional divisor of the contraction $q_i$, i.e. the minimal elliptic subcurve of $(\B_{i})_b$. We claim that
\begin{align*}
 p_i^*\omega_{\C_{i}/B}&=\omega_{\B_{i}/B}(-\Sigma R_i), \\
 q_i^*\omega_{\C_{i+1}/B}&=\omega_{\B_{i}/B}(E_i).
\end{align*}
The first formula is clear since $p_i$ is a simple blow-up. For the second formula, note that $q_i^*\omega_{\C_{i+1}/B}=\omega_{\B_{i}/B}(D)$ where $D$ is the unique Cartier divisor supported on $E_i$ such that $\omega_{\B_{i}/B}(D)|_{E_i} \simeq \O_{E_i}$. Clearly, $D=E_i$ since $\omega_{\B_{i}/B}(E_i)|_{E_i} \simeq \omega_{E_i} \simeq \O_{E_i}$. From these formulas, it follows that 
\begin{align*}
g^{i}_*\omega_{\C_{i}/B}&=h^{i}_*\omega_{\B_{i}/B},\\
g^{i+1}_*\omega_{\C_{i+1}/B}&=h^{i}_*\omega_{\B_{i}/B}(E_i).
\end{align*}
Thus, to compare $g^{i}_*\omega_{\C^i/B}$ and $g^{i+1}_*\omega_{\C^{i+1}/B}$, we consider the exact sequence on $\B_{i}$:
$$
0 \rightarrow \omega_{\B_{i}/B} \rightarrow \omega_{\B_{i}/B}(E_i) \rightarrow \O_{E_i} \rightarrow 0.
$$
Pushing forward, we obtain
$$
0 \rightarrow h^{i}_*\omega_{\B_{i}/B} \rightarrow h^{i}_*\omega_{\B_{i}/B}(E_i) \rightarrow h^{i}_*\O_{E_i} \rightarrow 0,
$$
where we have used the fact that the connecting homomorphism $h^{i}_*\O_{E_i} \rightarrow R^1h^{i}_*\omega_{\B_{i}/B}$ is zero, since $h^{i}_*\O_{E_i}\simeq k(b)$ is torsion, while $R^1h^{i}_*\omega_{\B_{i}/B}$ is locally free. We conclude that
$$
c_1 (h^{i}_*\omega_{\B_{i}/B}) = c_1( h^{i}_*\omega_{\B_{i}/B}(E_i))+1,
$$
which implies
$$
c_1 (g^{i}_*\omega_{\C_{i}/B}) = c_1( g^{i+1}_*\omega_{\C_{i+1}/B})+1,
$$
as desired.

To prove the formula relating $\psi^l.B$ and $\psi^m.B$, let us define $\{\sigma_i^j\}_{i=1}^{n}$ to be the strict transform of the sections  $\{\sigma_i\}_{i=1}^{n}$ on $\C_{j}$. Since the desingularization maps $\C_0 \rightarrow \D$ and $\C_k \rightarrow \C$ are isomorphisms in a neighborhood of the sections and hence do not effect the sum of the self-intersections, we have
\begin{align*}
\psi^{l}.B&=-\sum_{i=1}^{n} (\sigma_i^0)^2,\\
\psi^{m}.B&=-\sum_{i=1}^{n} (\sigma_i^k)^2.
\end{align*}
Thus, it suffices to show that for $j=0, 1, \ldots, k-1$
$$
\sum_{i=1}^{n}(\sigma_i^{j+1})^2-\sum_{i=1}^{n}(\sigma_i^{j})^2=-n_j.
$$
To see this, simply note that blow-up $p_j$ is supported along $n_j$ marked points, and the self-intersections of the strict transforms of the corresponding sections each decrease by one. On the other hand, the contraction $q_j$ is an isomorphism in a neighborhood of the sections and hence does not affect their self-intersections.

To prove the formula relating $\delta_0^l.B$ and $\delta_0^m.B$, let us define $\delta^i$ to be the number of disconnecting nodes in the fibers of $\C_i \rightarrow B$. Since the desingularization maps $\C_0 \rightarrow \D$ and $\C_k \rightarrow \C$ introduce $d-1$ nodes into the special fiber for each node counted with multiplicity $d$ in $\delta^l_0.B$ and $\delta^m_0.B$ respectively, we have
\begin{align*}
\delta_0^l.B&=\delta^0,\\
\delta_0^m.B&=\delta^k.
\end{align*}
Thus, it suffices to show that for $i=0, 1, \ldots, k-1$
$$
\delta^{i+1}-\delta^i=-m_i.
$$
To see this, note that the blow-up $p_i$ introduces $n_i$ disconnecting nodes into the special fiber, but the contraction $q_i$ absorbs $n_i+m_i$ disconnecting nodes into an elliptic $(n_i+m_i)$-fold point. Thus, there are $m_i$ fewer nodes in $(\C_{i+1})_b$ than in $(\C_i)_b$.

The final formula is an obvious consequence of the preceding two.
\end{proof}

This analysis clearly extends to the case when $\D \dashrightarrow \C$ is an isomorphism away from multiple fibers, since we can perform the necessary blow-ups and contractions on each fiber individually.
\begin{corollary}\label{C:CaseIIFormulas}
Suppose that $(f: \C \rightarrow B, \sigman)$ is a family of $m$-stable curves and $(g: \D \rightarrow B, \sigman)$ is a family of $l$-stable curves with $l<m$. Suppose that the generic fiber of $f$ has no disconnecting nodes, and that there is a birational morphism $\D \dashrightarrow \C$, so that $\D$ and $\C$ are isomorphic away from the fibers over $b_1, \ldots, b_t \in B$. Then we have
\begin{align*}
\lambda^{m}.B&=\lambda^{l}.B+\sum_{i=1}^{t}k_i\\
(\psi^m-\delta_{0}^m).B&=(\psi^l-\delta_{0}^l).B+\sum_{i=1}^{t}\sum_{j=1}^{k_i}l_{ij}
\end{align*}
where $k_i$ is the number of blow-ups/contractions required to transform the fiber $\D_{b_i}$ into $\C_{b_i}$, and $l_{ij}$ is the level of the elliptic bridge contracted in the $j^{th}$ step of this transformation.
\end{corollary}
\begin{proof}
Immediate from Proposition \ref{P:IntersectionNumbers}.
\end{proof}

\section{Proof of main results}\label{S:MainTheorem}

\subsection{The birational contraction $\phi:\M_{1,n} \dashrightarrow \M_{1,n}(m)^*$}\label{S:Discrepancy}
Recall that if $\phi:X \dashrightarrow Y$ is a birational map between normal algebraic spaces, we say that $\phi$ is a \emph{birational contraction} if  $\Exc(\phi^{-1})$ has codimension $\geq 2$. The \emph{exceptional divisors} of $\phi$ are the divisors on $X$ whose birational image in $Y$ has codimension $\geq 2$.
\begin{lemma}
$\phi: \M_{1,n} \dashrightarrow \NM{n}{m}$ is a birational contraction with exceptional divisors $\{ \Delta_{0,S}\}_{S \in [n]_{n-m+1}^{n}}$.
\end{lemma}
\begin{proof}
Consider the commutative diagram
\[
\xymatrix{
\M_{1,n} \ar@{-->}[r]^{\phi} \ar@{-->}[dr] &\M_{1,n}(m)^* \ar[d]^{\pi}\\
&\M_{1,n}(m)\\
}
\]
Let $\E_{l} \subset \M_{1,n}(m)$ denote the locally closed subspace parametrizing curves with an elliptic $l$-fold point. Since an $m$-stable curve is stable iff it is nodal, the open-set $\U:=\M_{1,n}(m)-\bigcup_{l=1}^{m} \E_{l}$ parametrizes stable curves, so $(\pi \circ \phi)^{-1}|_{\U}$ is an isomorphism. Thus, $\Exc(\phi^{-1}) \subset \bigcup_{i=1}^{m} \pi^{-1}(\E_{l})$. Since $\pi$ is finite, $\bigcup_{i=1}^{m} \pi^{-1}(\E_{l})$ has codimension $\geq 2$ by Corollary \ref{C:DimStrata}.

To see that $\Exc(\phi) \subset \{ \Delta_{0,S}\}_{S \in [n]_{n-m+1}^{n}}$ simply observe that the generic point of each divisor  $\{ \Delta_{0,S}\}_{S \in [n]_{2}^{n-m}}$ corresponds to an $m$-stable curve so that $\phi$ must be an isomorphism at this point. Conversely, the generic point of each divisor $\{ \Delta_{0,S}\}_{S \in [n]_{n-m+1}^{n}}$ is not $m$-stable and is replaced by an $m$-stable curve with an elliptic $l$-fold point, where $l=n-|S|+1$. Thus, the birational images of  $\{ \Delta_{0,S}\}_{S \in [n]_{n-m+1}^{n}}$ are contained in $\bigcup_{l=1}^{m} \pi^{-1}(\E_{l})$, which has codimension $\geq 2$.
\end{proof}

In order to make calculations with test curves, it will be necessary to have a precise description of the locus on which $\phi$ is regular. The following lemma gives a useful tool for determining this locus.

\begin{lemma}\label{L:Regular}
Suppose $\phi:X \dashrightarrow Y$ is a birational map of proper algebraic spaces with $X$ normal, and suppose $U \subset X$ is an open subset such that $\phi|_{U}$ is an isomorphism. If $x \in X$ is any point, then $\phi$ is regular at $x$ iff there exists a point $y \in Y$ such that the following condition holds:

For any map $t: \Delta \rightarrow X$ satisfying
\begin{enumerate}
\item $\Delta$ is the spectrum of a DVR with generic point $\eta \in \Delta$ and closed point $0 \in \Delta$,
\item $t(\eta) \in U$,
\item $t(0)=x$,
\end{enumerate}
the composition $\phi \circ t: \Delta \rightarrow Y$ satisfies $\phi \circ t (0)=y$. (The composition $\phi \circ t$ is regular, since $Y$ is proper.)
\end{lemma}
\begin{proof}
The existence of a point $y \in Y$ satisfying the given condition is clearly necessary for $\phi$ to be regular at $x$. We will prove that it is sufficient. Consider a resolution of the rational map $\phi$:
\[
\xymatrix{
&W \ar[dr]^{q} \ar[dl]_{p}&\\
X \ar@{-->}[rr]^{\phi}&&Y\\
}
\]
We may choose the resolution so that $W$ is normal and $p$ and $q$ are isomorphisms when restricted to $p^{-1}(U)$.

We claim that $p^{-1}(x) \subset q^{-1}(y)$. Given any point $w \in p^{-1}(x)$, the fact that $p^{-1}(U) \subset W$ is dense implies there exists a map $t: \Delta \rightarrow W$ such that $t(\eta) \in p^{-1}(U)$ and $t(0)=w$. Clearly, the composition $p \circ t$ satisfies the conditions (1), (2), and (3), so our hypothesis ensures that $\phi \circ p \circ t(0) = q \circ t(0) = y \in Y$. Thus, $w \in q^{-1}(y)$ as desired.

Now if $Y$ is normal, then $q$ factors through $p$ and we are done. If $Y$ is not normal, let $\pi: \tilde{Y} \rightarrow Y$ be the normalization of $Y$ and let $\pi^{-1}(y)=\{y_1, \ldots, y_m\}$. Since $W$ is normal, $q$ factors through $\pi$, say $q=\pi \circ \tilde{q}$, and $q^{-1}(y)=\tilde{q}^{-1}(y_1) \cup \ldots \cup \tilde{q}^{-1}(y_m)$. By Zariski's main theorem $p^{-1}(x)$ is connected so the above argument gives $p^{-1}(x) \subset q^{-1}(y_i)$ for some $i$. Thus, $\tilde{q}$ factors through $p$ and $\tilde{\phi}: X \dashrightarrow \tilde{Y}$ is regular at $x \in X$ with $\tilde{\phi}(x)=y_i$. Since $\pi$ is regular at $y_i$, the composition $\phi=\pi \circ \tilde{\phi}$ is regular at $x$, as desired.
\end{proof}
\begin{corollary}\label{C:RegularLocus}
The birational map $\phi:\M_{1,n} \dashrightarrow \M_{1,n}(m)$ is regular at $[C, \pn] \in \M_{1,n}$ iff $[C, \pn]$ satisfies one of the following conditions:
\begin{enumerate}
\item $[C, \pn] \notin \Delta_{0,S}$ for any $S \in [n]_{n-m+1}^{n}$, or
\item $C$ has only one disconnecting node.
\end{enumerate}
\end{corollary}
\begin{proof}
If $[C, \pn] \notin \Delta_{0,S}$ for some $S \in [n]_{n-m+1}^{n}$, then $[C, \pn]$ is $m$-stable so $\phi$ is obviously regular in a neighborhood of $[C, \pn]$. Thus, we may assume that $[C, \pn] \in \Delta_{0,S}$ for some $S \in [n]_{n-m+1}^{n}$ and that $C$ has exactly one disconnecting node.

By Lemma \ref{L:Regular}, it suffices to show that there exists a point $[C', \{p_i'\}_{i=1}^{n}] \in \M_{1,n}(m)$ with the property that, for any map $t:\Delta \rightarrow \M_{1,n}$ such that $t(\eta) \in M_{1,n}$ and $t(0)=[C, \pn]$, we have $\phi \circ t(0)=[C', \{p_i'\}_{i=1}^{n}]$. Write 
$$(C, \pn)=(E, \{p_i\}_{i \in [n] \backslash S}, q_2) \cup_{q_1 \sim q_2} (\P^{1}, \{p_i\}_{i \in S}, q_1),$$ 
where $E$ and $\P^1$ are the two connected components of the normalization of $C$ at its unique disconnecting node. Now let $(C', \{p_i'\}_{i=1}^{n})$ be the unique isomorphism class of elliptic $(n-|S|+1)$-fold pointed curve with normalization equal to 
$$\left( \coprod _{i \in [n] \backslash S} (\P^{1}, p_i, q_i) \right) \coprod (\P^{1}, \{p_i\}_{i \in S}, q_1) ,$$
where $(\P^{1}, p_i, q_i) \simeq (\P^1, 0, \infty)$ and we identify $q_1 \cup \{q_i \}_{i \in [n] \backslash S}$ to form an elliptic $n-|S|+1$-fold point. Note that, by Lemma 2.7, the attaching data for this elliptic $n-|S|+1$-fold point is uniquely determined. We claim that $[C', \{p_i'\}_{i=1}^{n}] \in \M_{1,n}(m)$ satisfies the desired condition.

Given a map $t:\Delta \rightarrow \M_{1,n}$ such that $t(\eta) \in M_{1,n}$ and $t(0)=[C, \pn]$, we may assume (after a finite base change) that $t$ corresponds to smoothing $(\C \rightarrow \Delta, \sigman)$ and it suffices to show that the $m$-stable limit of the generic fiber $\C_{\eta}$ is $[C', \{p_i'\}_{i=1}^{n}]$.  To check this, we use the explicit algorithm for finding $m$-stable limits as described in \cite[Theorem 3.11]{SmythEI}. When the total space of $\C$ is smooth, the $m$-stable limit is produced simply by blowing up $\C$ at the marked points on $E$ and contracting the strict transform of $E$, which precisely gives $[C', \{p_i'\}_{i=1}^{n}]$. If $\C$ has an $A_{k}$ singularity at the unique disconnecting node of $C$, the $m$-stable limit is produced by desingularizing $\C$ at this point, and repeating this blow-up/contraction process $k+1$ times. We leave it to the reader to check that the result is again simply $[C', \{p_i'\}_{i=1}^{n}]$.

To see that if $[C, \pn] \in \M_{1,n}$ fails to satisfy (1) and (2), then $\phi$ is not regular at $[C, \pn]$ it suffices to exhibit two smoothings of $[C, \pn]$ which have different $m$-stable limits. We leave this as an exercise for the reader.
\end{proof}

\begin{corollary}\label{C:ContractedCurve}
Suppose $[C, \pn] \in \M_{1,n}$ satisfies
\begin{enumerate}
\item $C \in \Delta_{0,S}$ for  some $S \in [n]_{n-m+1}^{n}$,
\item $C$ has exactly one disconnecting node.
\end{enumerate}
If we write
$$(C, \pn)=(E, \{p_i\}_{i \in [n] \backslash S}, q_2) \cup_{q_1 \sim q_2} (\P^{1}, \{p_i\}_{i \in S}, q_1),$$ then $\phi([C, \pn])=[C', \{p_i'\}_{i=1}^{n}]$, where $(C', \{p_i'\}_{i=1}^{n})$ be the unique isomorphism class of elliptic $(n-|S|+1)$-fold pointed curve with normalization equal to $(\P^{1}, \{p_i\}_{i \in S}, q_1) \cup \coprod _{i \in [n] \backslash S} (\P^{1}, p_i, q_i)$. 
\end{corollary}
\begin{proof}
Immediate from the  proof of the preceding corollary.
\end{proof}

\begin{corollary}\label{C:Regularity}
The birational map $\phi: \M_{1,n}(m-1) \dashrightarrow \M_{1,n}(m)$ is regular iff $m=1$ or $m=n-1$.
\end{corollary}
\begin{proof}
If $m=1$, then any $m$-stable curve evidently satisfies condition (1) in Corollary \ref{C:RegularLocus}. If $m=n-1$, then any $m$-stable curve has at most one disconnecting node, i.e. any $m$-stable curve satisfies condition (2) in Corollary \ref{C:RegularLocus}. On the other hand, if $2 \leq m \leq n-2$, then the reader may easily check that there exist $m$-stable curves which fail to satisfy both (1) and (2).
\end{proof}

Since $\phi: \M_{1,n} \dashrightarrow \NM{n}{m}$ is a birational contraction, push forward of cycles and pull back of divisors induce well-defined maps:
\begin{align*}
\phi_*: N^{1}(\M_{1,n}) \rightarrow N^1(\NM{n}{m}),\\
 \phi^*: N^1(\NM{n}{m}) \rightarrow N^{1}(\M_{1,n}),
\end{align*}
where $N^{1}(X)$ denotes the $\Q$-vector space generated Cartier divisors moduli numerical equivalence. Since $N^1(\M_{1,n})$ and $N^1(\NM{n}{m})$ are generated by the classes of the boundary divisors (Propositions 3.1 and 3.2), the following proposition determines $\phi_*$ and $\phi^*$ completely.

\begin{proposition}\label{P:PushPull}
For the birational contraction
$
\phi:\M_{1,n} \dashrightarrow \NM{n}{m},
$
$\phi_*$ and $\phi^*$ satisfy the following formulas:
\begin{align*}
\phi_{*}\Delta_{0,S}&=\Delta_{0,S}\\
\phi_{*}\Delta_{irr}&=\Delta_{irr}\\
\phi^*\Delta_{0,S}&=\Delta_{0,S}\\
\phi^*\Delta_{irr}&=\Delta_{irr}+\sum_{S \in [n]_{n-m+1}^{n}}12\Delta_{0,S}
\end{align*}
\end{proposition}
\begin{proof}
The push forward formulae are immediate from the definition of $\Delta_{0,S}$ and $\Delta_{irr}$. For the pull back formulae, note that since $\phi$ has exceptional divisors $\{\Delta_{0,T}\}_{T \in [n]_{n-m+1}^{n}}$ we may write
\begin{align*} \tag{\dag}
\phi^*\Delta_{irr}=\Delta_{irr}+\sum_{T \in [n]_{n-m+1}^{n}}a_T\Delta_{0,T},\\ \tag{\ddag}
\phi^*\Delta_{0,S}=\Delta_{0,S}+\sum_{T \in [n]_{n-m+1}^{n}}b_T\Delta_{0,T},
\end{align*}
for some coefficients $a_{T},b_{T}$. We will prove that $a_{T}=12$ and $b_{T}=0$ by intersecting with an appropriate collection of test curves.

Fix $T \in [n]_{n-m+1}^{n}$, and define a complete one-parameter family of $n$-pointed stable curves as follows: Let $(\C_1 \rightarrow B_{T}, \{\sigma_i\}_{i=1}^{|T|+1})$ be a non-constant family of $(|T|+1)$-pointed stable curves of genus one, with smooth general fiber and only irreducible singular fibers. The existence of such families follows from Corollary \ref{C:TestCurves} in Section 4.2. (The reader may verify that this Proposition is not invoked in the proof of any intermediate results.) Let $\sigma_1, \ldots, \sigma_{|T|}$ be labeled by the elements of $T$, and consider $\sigma_{|T|+1}$ as an attaching section. Next, let $(\C_2 \rightarrow B_{T}, \{\tau_i\}_{i=1}^{n-|T|+1})$ be a constant family of smooth rational curves over $B_{T}$ with $n-|T|+1$ constant sections. Let $\tau_1, \ldots, \tau_{n-|T|}$ be labeled by elements of $[n] \backslash T,$ and consider $\tau_{n-|T|+1}$ as an attaching section. Gluing $\C_1$ to $\C_2$ along $\sigma_{|T|+1} \sim \tau_{n-|T|+1}$, we obtain a family of $n$-pointed stable curves over $B_{T}$. We claim that the curve $B_{T} \subset \M_{1,n}$ satisfies
\begin{itemize}
\item[(1)] $\phi$ is regular in a neighborhood of $B_T$,
\item[(2)] $B_{T}$ is contracted by $\phi$,
\item[(3)] $\Delta_{irr}.B_{T}=-12 (\Delta_{0,T}.B_{T})$
\item[(4)] $\Delta_{0,S}.B_{T}=0$ if $S \neq T$.
\end{itemize}
Part (1) follows from Corollary \ref{C:RegularLocus}, since each fiber of the family has only one disconnecting node. Using Corollary \ref{C:ContractedCurve}, one sees that each point of $B_{T}$ is mapped to the same point in $\M_{1,n}(m)$, and (2) follows. Parts (3) is a standard calculation using the relations in $\Pic_{\Q}(\M_{1,n})$ (Proposition 3.1), and (4) is immediate from the construction. Intersecting both sides of $\dag$ and $\ddag$ with the test curve $B_{T}$ gives $a_T=12$ and $b_T=0$, as desired.
\end{proof}
We can use our formulas for $\phi_*$ and $\phi^*$ to compare section rings on $\M_{1,n}$ and $\NM{n}{m}$.

\begin{proposition}\label{P:Discrepancy}
$
R(\M_{1,n}, D(s))=R(\NM{n}{m}, \phi_*D(s))
$ iff $s \leq 12-m$.
\end{proposition}
\begin{proof}
It suffices to show that $\phi^*\phi_*D(s)-D(s) \geq 0$ iff $s \leq 12-m$. Using the relations in $\Pic_{\Q}(\M_{1,n})$ (Proposition \ref{P:StartingPicardGroup}), we have
$$
D(s):=s\lambda+\psi-\Delta=\frac{(n+s-12)}{12}\Delta_{irr}+\sum_{S \in [n]_{2}^{n}}(|S|-1)\Delta_{0,S}
$$
Using the formulae of Proposition \ref{P:PushPull}, we have
\begin{align*}
\phi_*D(s)&=\frac{(n+s-12)}{12}\Delta_{irr}+\sum_{S \in [n]_{2}^{n-m}}(|S|-1)\Delta_{0,S},\\
\phi^*\phi_*D(s)&=\frac{(n+s-12)}{12}\Delta_{irr}+\sum_{S \in [n]_{2}^{n-m}}(|S|-1)\Delta_{0,S}+\sum_{S \in [n]_{n-m+1}^{n}}(n+s-12)\Delta_{0,S}.
\end{align*}
Thus, 
$$
D(s)-\phi_*\phi_*D(s)=\sum_{S \in [n]_{n-m+1}^{n}}(|S|+11-n-s) \Delta_{0,S}.
$$
Since $|S| \geq n-m+1$, we have $D(s)-\phi_*\phi_*D(s) \geq 0 \iff 12-m-s \geq 0$. Thus, $s  \leq 12-m$ iff
$
R(\M_{1,n}, D(s))=R(\NM{n}{m}, \phi_*D(s)).
$
\end{proof}

\subsection{Ample Divisors on $\M_{1,n}(m)$}\label{S:AmpleDivisors}
In this section, we prove that $\M_{1,n}(m)$ is projective. More precisely, we show that
$$
\psi-\delta_0-s\lambda \text{ is ample on } \M_{1,n}(m) \text{ if } m < s < m+1.
$$
In conjunction with the discrepancy calculation of Propostion \ref{P:Discrepancy}, this will allow us to prove our main result (Corollary \ref{C:MainResult}). Our proof of ampleness proceeds via Kleiman's criterion, i.e. we will show that the given divisors have positive intersection on all curves in  $\M_{1,n}(m)$. We begin with two preparatory lemmas.
\begin{lemma}\label{L:BasicNefness}
\begin{itemize}
\item[]
\item[(1)] $\lambda$ is nef on $\M_{1,n}$,
\item[(2)] $\psi-\delta$ is ample on $\M_{0,n}$,
\item[(3)]$\psi-\delta_0-\lambda$ is nef on $\M_{1,n}$,
\item[(3)] $\psi_i$ is nef on $\M_{0,n}$ for each $i=1, \ldots, n$.
\end{itemize}
\end{lemma}
\begin{proof}
(1) and (2) are well-known. For (2), consider the closed immersion
$$i:\M_{0,n} \rightarrow \M_{g}$$
defined by attaching fixed curves of genus $g_1, \ldots, g_n \geq 2$ to the $n$ marked points, where $g$ is chosen so that $g_1+ \ldots +g_n=g$. By \cite[Theorem 1.3]{CH}, the divisor $s\lambda-\delta$ is ample on $\M_{g}$ if $s>11$. Since
\begin{align*}
&i^*\lambda=0,\\
&i^*\delta=\delta-\psi,
\end{align*}
we conclude that $i^*(12\lambda-\delta)=\psi-\delta$ is ample on $\M_{0,n}$.

The proof of (3) is similar. Consider the closed immersion $$i:\M_{1,n} \rightarrow \M_{g}$$
defined by attaching fixed curves of genus $g_1, \ldots, g_n \geq 2$ to the $n$ marked points, where $g$ is chosen so that $g_1+ \ldots +g_n+1=g$. Using the same formulae as above and the relation $\delta_{irr}=12\lambda$ on $\M_{1,n}$, one checks that
$$
i^*(11\lambda+\psi-\delta)=\psi-\delta_0-\lambda,
$$
so $\psi-\delta_0-\lambda$ is nef on $\M_{1,n}$.
\end{proof}

For our second lemma, suppose $(f:\C \rightarrow B, \sigman)$ is a family of $m$-stable curves over a smooth curve $B$ and that every fiber of $f$ contains an elliptic $l$-fold point, for some $l \geq 1$. Then $f$ admits a section $\tau$ such that $\tau(b) \in \C_{b}$ is an elliptic $l$-fold point for all $b \in B$, and we may consider the normalization $\tilde{\C} \rightarrow \C$ along $\tau$. Let $\{\tilde{\tau}_i\}_{i=1}^{l}$ be the sections lying over $\tau$, and let $S_i$ be the subset of marked points lying on the $i^{th}$ connected component of the normalization. The normalization $\tilde{\C}$ decomposes as:
$$
\coprod_{i=1}^{l}(\tilde{\C}_i, \tilde{\tau}_i, \{\tilde{\sigma}_j\}_{j \in S_i}),
$$
where each $(\tilde{\C}_i, \tilde{\tau}_i, \{\tilde{\sigma}_j\}_{j \in S_i})$ is a family of semistable genus zero curves over $B$. If we assume, in addition, that the generic fiber of $\C$ has no disconnecting nodes, then the generic fiber of each $\tilde{\C}_i$ is smooth. In this case, for each $i$ satisfying $|S_i| \geq 2$, there is a well-defined stabilization map, i.e. a birational map $\tilde{\C}_i \rightarrow \tilde{\C}_i^s$ obtained by blowing down the semistable components in the fibers of $\tilde{\C}_i$. Let $\tilde{\tau}_i^s$ and $\tilde{\sigma}_i^{s}$ be the images of $\tilde{\tau}_i$ and $\tilde{\sigma_i}$ under this map.

Without loss of generality, we may assume the $S_i$ are ordered so that $|S_i| \geq 2$ for $i=1, \ldots, k$ and $|S_{k+1}|=|S_{k+2}|=\ldots=|S_{l}|=1$. Then, for each $i=1, \ldots, k$, each $(\tilde{\C}^s_i, \tilde{\tau}_i^s, \{\tilde{\sigma}^s_j\}_{j \in S_i})$ is a stable family of genus zero curves over $B$, so we have a map
$$
c^{s}:B \rightarrow \M_{0,|S_i|+1} \times \ldots \times \M_{0,|S_k|+1},
$$
and we may define
\begin{align*}
\psi^{s}.B:=\deg_{B}(c^{s})^*\psi\\
\delta_0^{s}.B:=\deg_{B}(c^{s})^{*}\delta.
\end{align*}
The following lemma compares the intersection number $(\psi-\delta_0).B$ with the intersection number $(\psi^{s}-\delta^{s}_0).B$.

\begin{lemma}\label{L:StableComparison}
Suppose $(f: \C \rightarrow B, \sigman)$ is a family of $m$-stable curves satisfying
\begin{enumerate}
\item Every fiber of $\C$ has an elliptic $l$-fold point, for some $l \geq 1$,
\item The generic fiber of $\C$ has no disconnecting nodes.
\end{enumerate}
With notation as above, we have
$$(\psi-\delta_0).B=(\psi^s-\delta^{s}_0).B+l\lambda.B,$$
\end{lemma}
\begin{proof}
As in the discussion preceding the lemma, we have a diagram
\[
\xymatrix{
&\tilde{\C} \ar[dr]^{\pi} \ar[dl]_{\phi}&\\
\tilde{\C}^{s}&&\C\\
}
\]
where $\pi$ is the normalization of $\C$ along $\tau$, and $\phi$ is the birational stabilization map. 

Since $\pi$ is an isomorphism in an open neighborhood of every node and every section $\sigma_i$, $\pi$ does not effect the relevant intersection numbers, i.e. we have
\begin{align*}
\delta_0.B&=\#\{\text{Nodes in fibers of $\C$}\}=\#\{\text{Nodes in fibers of $\tilde{\C}$}\},\\
\psi.B&=-\sum_{i=1}^{k}\sigma_i^2=-\sum_{i=1}^{k}\tilde{\sigma}_i^2,
\end{align*}
where the nodes are counted with suitable multiplicity.

To analyze the effect of $\phi$ on these intersection numbers, observe that if $R \simeq \P^{1}$ is a component of a fiber of $\tilde{\C}_i$ contracted by $\phi$, then $R$ meets the rest of the fiber at a single node and the section $\tilde{\tau}_i$ passes through $R$. If the attaching node is an $A_{k}$-singularity of the total space, then blowing down $R$ decreases the number of nodes in $\tilde{\C}$ by $k$ (counted with multiplicity), while raising the self-intersection of the section $\tilde{\tau}_i$ by $k$. Thus,
\begin{align*}
\psi^{s}.B-\delta_0^{s}.B&=-\sum_{i=1}^{n}(\tilde{\sigma}^s_i)^2-\sum_{i=1}^{l}(\tilde{\tau}^{s}_i)^2-\#\{\text{Nodes in fibers of $\tilde{\C}^{s}$}\}\\
&=-\sum_{i=1}^{n}\tilde{\sigma}_i^2-\sum_{i=1}^{l}\tilde{\tau}_i^2-\#\{\text{Nodes in fibers of $\tilde{\C}$}\}\\
&=\psi.B-\delta_0.B-\sum_{i=1}^{l}(\tilde{\tau}_i)^2
\end{align*}
Applying Proposition \ref{P:LambdaNumber}, we see that the last line is equivalent to $(\psi-\delta_0).B-l\lambda.B$, as desired.
\end{proof}

\begin{proposition}\label{P:Nefness}
If $s \in \Q \cap [m,m+1]$, then $\psi-\delta_0-s\lambda$ is nef on $\M_{1,n}(m)$. 
\end{proposition}
\begin{proof}
Fix $s \in \Q \cap [m,m+1]$. To prove that $\psi-\delta_0-s\lambda$ is nef on $\M_{1,n}(m)$, it suffices to show that $\psi-\delta_0-s\lambda$ has non-negative degree on any family of $m$-stable curves $(f: \C \rightarrow B, \sigman)$ over a smooth curve $B$. We begin with three reductions.
\begin{reductionI}
We may assume that the generic fiber of $\C$ has no disconnecting nodes.
\end{reductionI}
\begin{proof}
We may decompose a generic fiber of $\C$ as
$$C=E \cup R_1 \cup \ldots \cup R_k,$$
where $E$ is the minimal elliptic subcurve of $C$, and $R_1, \ldots, R_k$ are rational tails meeting $E$ in a single node \cite[Lemma 3.1]{SmythEI}. Since the limit of a disconnecting node is a disconnecting node, there exist sections $\tau_1, \ldots, \tau_k:B \rightarrow \C$ such that 
\begin{enumerate}
\item[(1)] $\tau_i(b) \in C_b$ is a disconnecting node for all $b \in B$.
\item[(2)] $\tau_i(b) \in E \cap R_i$ over the generic point of $B$.
\end{enumerate}
Let $\tilde{\C} \rightarrow \C$ be the normalization of $\C$ along $\cup_{i=1}^{k}\tau_i$, so we have
$$\tilde{\C}=\E \coprod \R_1 \coprod \ldots \coprod \R_k,$$
where  $\E \rightarrow B$ is a family of genus one curves and each $\R_i \rightarrow B$ is a family of genus zero curves. Mark the two sections of $\tilde{\C}$ lying above $\tau_i$ as $\tau_i'$ and $\tau_i''$, so that $(\tilde{\C}, \sigman, \{\tau_i'\}_{i=1}^{k}, \{\tau_i''\}_{i=1}^{k})$ decomposes as

$$(\E, \{\sigma_i\}_{i \in S_0}, \{\tau_i'\}_{i=1}^{k}) \coprod (\R_1, \{\sigma_i\}_{i \in S_1}), \tau_1'') \coprod \ldots \coprod (\R_k, \{\sigma_i\}_{i \in S_k}), \tau_k''),$$
where $\{S_0, S_1, \ldots, S_k\}$ is some partition of $[n]$. Note that  $(\E, \{\sigma_i\}_{i \in S_0}, \{\tau_i'\}_{i=1}^{k})$ is an $(|S_0|+k)$-pointed $m$-stable curve, and each $(\R_j, \{\sigma_i\}_{i \in S_j}), \tau_j'')$ is an $(|S_j|+1)$-pointed stable curve of genus zero. Let $c_0:B \rightarrow \S_{1,|S_0|+k}(m)$ and $c_j: B \rightarrow \S_{0,|S_j|+1}$ be the corresponding classifying maps, and define
\begin{align*}
\lambda^i.B:&=\deg_{B} c_i^*\lambda\\
(\psi-\delta_0)^i.B:&=\deg_B c_i^*(\psi-\delta_0)
\end{align*}

Since the degree of $\lambda$ is zero on any family of genus zero stable curves, we have
\begin{align*}
\lambda.B&=\lambda^0.B+\sum_{j=1}^{k} \lambda^j.B=\lambda^0.B
\end{align*}
Furthermore, since $(\psi-\delta)$ is ample on $\M_{0,n}$ (Lemma \ref{L:BasicNefness}), we have
\begin{align*}
(\psi-\delta_0).B&=(\psi-\delta_0)^0.B+\sum_{j=1}^{k}(\psi-\delta_0)^j.B>(\psi-\delta_0)^0.B.
\end{align*}
Altogether, we obtain
$$
(\psi-\delta_0-s\lambda).B>(\psi-\delta_0-s\lambda)^0.B.
$$
Since $(\E \rightarrow B, \{\sigma_i\}_{i \in S_0}, \{\tau_i'\}_{i=1}^{k})$  is an $m$-stable curve with no disconnecting nodes in the generic fiber, it suffices to prove the non-negativity of $\psi-\delta_0-s\lambda$ on families $m$-stable curves satisfying this extra condition.
\end{proof}
\begin{reductionII}
We may assume that $\lambda.B<0$.
\end{reductionII}
\begin{proof}
Using the relations in Proposition \ref{P:PicardGroup}, we have
$$
(\psi-\delta_0-s\lambda).B=\sum_{S \subset [n]_2^{n-m}}(|S|-1)\delta_{0,S}.B+(n-s)\lambda.B.
$$
By the first reduction, we have $\delta_{0,S}.B>0$ for each $S \subset [n]_2^{n-m}.$ Furthermore, $n-s \geq 0$, since $s \leq m+1$ and $m \leq n-1$. Thus, if $\lambda.B \geq 0$, the intersection number $(\psi-\delta_0-s\lambda).B$ is non-negative. 
\end{proof}
\begin{reductionIII}
We may assume the generic fiber of $\C$ contains an elliptic $l$-fold point, for some $l \geq 1$.
\end{reductionIII}
\begin{proof}
Since $\lambda$ is nef on $\M_{1,n}$ (Lemma \ref{L:BasicNefness}), Corollary \ref{C:CaseIIFormulas} (applied with $l=0$) implies that $\lambda.B \geq 0$ for any $m$-stable curve with nodal generic fiber. Thus, by the second reduction, we may assume that the generic fiber of $\C$ contains an elliptic $l$-fold point, for some $l \geq 1.$
\end{proof}

Now suppose that \emph{every} fiber of $\C$ contains an elliptic $l$-fold point. In this case, Lemma \ref{L:StableComparison} implies that
$$
(\psi-\delta_0-s\lambda).B=(\psi^s-\delta_0^s).B+(l-s)\lambda.B,
$$
where $(\psi^s-\delta_0^s).B$ is the sum of the intersection numbers of $\psi-\delta$ on the families of genus zero stable curves obtained by normalizing $\C$ along the locus of elliptic $l$-fold points, and stabilizing the resulting families of semistable curves. By Lemma \ref{L:BasicNefness} (2), $(\psi^s-\delta_0^s).B>0$, so
$$
(\psi-\delta_0-s\lambda).B>(l-s)\lambda.B.
$$
Since $l \leq m  \leq s$ and $\lambda.B<0$, this intersection number is non-negative.

It remains to consider the possibility that there is a finite set of points $b_1, \ldots, b_t \in B$ where the fibers of $\C$ acquire elliptic $k$-fold points with $k>l$. Since the restriction of $f$ to $B-\{b_1, \ldots, b_t\}$ is an $l$-stable curve, we have a classifying map $c_l:B \rightarrow \S_{1,n}(l)$ and we set
\begin{align*}
\lambda^l.B:=&\deg_{B}c_l^*\lambda,\\
(\psi-\delta_0)^l.B:=&\deg_Bc_l^*(\psi-\delta_0).
\end{align*}
In the preceding paragraph, we saw that
$$(\psi-\delta_0-s\lambda)^l.B>(l-s)\lambda^l.B$$
On the other hand, Corollary \ref{C:CaseIIFormulas} says
\begin{align*}
\lambda.B-\lambda^l.B&=\sum_{i=1}^{t}k_i,\\
(\psi-\delta_{0}).B-(\psi-\delta_0)^l.B&=\sum_{i=1}^{t}\sum_{j=1}^{k_i}l_{ij},
\end{align*}
where it takes $k_i$ blow-ups/contractions to transform the $l$-stable fiber $\C^l_{b_i}$ into the $m$-stable fiber $\C_{b_i}$, and $l_{ij}$ is the level of the elliptic bridge contracted at the $j^{th}$-step. Thus, we obtain
$$
(\psi-\delta_0-s\lambda).B-(\psi-\delta_0-s\lambda)^l.B=\sum_{i=1}^{t}\sum_{j=1}^{k_i}(l_{ij}-s).
$$
We have $l_{ij} \geq l+1$, since we only contract elliptic bridges of level $l+1, \ldots, m$ in transforming an $l$-stable fiber to an $m$-stable fiber. Thus, we obtain
$$
\sum_{i=1}^{t}\sum_{j=1}^{k_i}(l_{ij}-s) \geq \sum_{i=1}^{t}\sum_{j=1}^{k_i}(l+1-s)=(l+1-s)\sum_{i=1}^{t}k_i.
$$
Combining the preceding inequalities, we obtain
\begin{align*}
(\psi-\delta_0-s\lambda).B&=(\psi-\delta_0-s\lambda)^l.B+\sum_{i=1}^{t}\sum_{j=1}^{k_i}(l_{ij}-s)\\
&>(l-s)\lambda^l.B+(l-s+1)\sum_{i=1}^{t}k_i\\
&=(l-s)\left(\lambda^l.B+\sum_{i=1}^{t}k_i\right)+\sum_{i=1}^{t}k_i,\\
&=(l-s)\lambda.B+\sum_{i=1}^{t}k_i,
\end{align*}
which is non-negative since $l \leq m \leq s$ and $\lambda.B<0$.
\end{proof}

To upgrade from nefness to ampleness, we will use Kleiman's criterion \cite[Theorem 2.19]{Kol}. Unfortunately, Kleiman's criterion can fail for algebraic spaces \cite[Excercise 2.19.3]{Kol}. Thus, we must first show that Kleiman's criterion applies to $\M_{1,n}(m)^*$ without assuming a priori that $\M_{1,n}(m)^*$ is a scheme.

\begin{lemma}\label{L:Kleiman} 
Any divisor in the interior of the nef cone of $\M_{1,n}(m)^*$ is ample.
\end{lemma}
\begin{proof}
To show that Kleiman's criterion applies to $\M_{1,n}(m)^*$, we must show that for any irreducible subvariety
$$
Z \subset \NM{n}{m}
$$
there exists an effective Cartier divisor $E$ which meets $Z$ properly \cite[Lemma 4.9]{FedSmyth}. Since $\NM{n}{m}$ is $\Q$-factorial, it is enough to show that there exists an open affine subscheme of $\NM{n}{m}$ meeting $Z$.

Let $\pi: \NM{n}{m} \rightarrow \M_{1,n}(m)$ be the normalization map, and consider the stratification of $\NM{n}{m}$ induced by the equisingular stratification of $\M_{1,n}(m)$:
$$
\NM{n}{m}=\pi^{-1}(M_{1,n}) \coprod \pi^{-1}(\E_0) \coprod \ldots \coprod \pi^{-1}(\E_{m}).
$$
Using Proposition \ref{P:Ampleness} and induction on $m$, we may assume that $\NM{n}{m-1}$ is projective. Since the open set
$$
\pi^{-1}(M_{1,n}) \coprod \pi^{-1}(\E_0) \coprod \ldots \coprod \pi^{-1}(\E_{m-1}) \subset \NM{n}{m}
$$
is isomorphic to an open subset of $\NM{n}{m-1}$, every point has an open affine neighborhood. Thus, we may assume $Z \subset \pi^{-1}(\E_{m}).$

Evidently, it is sufficient to produce an effective Cartier divisor on $\M_{1,n}(m)$ which meets $\pi(Z)$ properly. $\pi(Z)$ lies in one of the irreducible components of $\E_{m}$, and by Proposition \ref{P:BoundaryStratification} these are each projective bundles of the form
$$
p: \P(\psi_1 \oplus \ldots \oplus \psi_k) \rightarrow \M_{0,|S_1|+1} \times \ldots \times \M_{0,|S_k|+1}.
$$
By construction, the divisor $\Delta_{0,S_1} \subset \M_{1,n}(m)$ restricts to a hyperplane subbundle
$$
\Delta_{0,S_1} \cap \P(\psi_1 \oplus \ldots \psi_k) \subset \P(\psi_1 \oplus \ldots \oplus \psi_k).
$$
If $Z$ meets $\Delta_{0,S_1}$ properly, we are done, since some multiple of $\Delta_{0,S_1}$ is Cartier. If not, then the map
$$
Z \rightarrow p(Z) \subset \M_{0,|S_1|+1} \times \ldots \times \M_{0,|S_k|+1}
$$
is finite. Since $M_{0,|S_1|+1} \times \ldots \times M_{0,|S_k|+1}$ is affine and $\dim p(Z)>1$, $p(Z)$ must meet some boundary divisor $\pi_i^*\Delta_{0,T}$, $T \subset S_i$. Equivalently, $Z$ meets the boundary divisor $\Delta_{0,T} \subset \M_{1,n}(m)$. Since some multiple of $\Delta_{0,T}$ is Cartier, we are done.
\end{proof}

Now we will upgrade our nefness result to an ampleness result by showing that $\psi-\delta_0-s\lambda$ remains ample under a small perturbation by boundary divisors.

\begin{proposition}\label{P:Ampleness} If $s \in \Q \cap (m,m+1)$, then $\psi-\delta_0-s\lambda$ is ample on $\M_{1,n}(m)$. In particular, $\M_{1,n}(m)$ is projective.
\end{proposition}
\begin{proof}
Fix $s \in \Q \cap (m,m+1)$. It is sufficient to show that $\pi^*(\psi-\delta_0-s\lambda)$ is ample, where $\pi:\NM{n}{m} \rightarrow \M_{1,n}(m)$ is the normalization map. By Proposition \ref{P:PicardGroup},
$$
\Pic(\NM{n}{m}) \otimes \Q=\Q\{\lambda, \delta_{0,S}: S \subset [n]_{2}^{n-m}\}.
$$
Thus, by Lemma \ref{L:Kleiman}, it is enough to show that there exists $c \in \Q_{>0}$ such that
$$
(\psi-\delta_0-s\lambda)+\epsilon_{\lambda}\lambda+\sum_{S \in [n]_2^{n-m}}\epsilon_{S}\delta_{0,S}
$$
is nef, for any choice of $\epsilon_{\lambda}, \epsilon_{S} \in \Q \cap (-c,c).$ Clearly, we may pick $c$ small enough that $(s-c, s+c) \in (m,m+1)$. Replacing $s$ by $s+\epsilon_{\lambda}$, it suffices to show that 
$$
(\psi-\delta_0-s\lambda)+\sum_{S \in [n]_2^{n-m}}\epsilon_{S}\delta_{0,S}
$$
is nef for any $\epsilon_{S} \in \Q \cap (-c,c).$ 

Since $\psi-\delta_0$ is ample on $\M_{0,n}$ (Lemma \ref{L:BasicNefness}), we may choose $c$ sufficiently small so that
\begin{itemize}
\item[(1)] $c<\frac{s}{m}-1$.
\item[(2)]  $(\psi-\delta_0)+\sum_{S \subset [k]_2^k}\epsilon_{S}\delta_{0,S}$  is ample on $\M_{0,k}$, for all $3 \leq k \leq n$ and $\epsilon_{S} \in (c,-c)$.
\end{itemize}
Now fix $c$ satisfying (1) and (2), and fix $\epsilon_{S} \in \Q \cap (-c,c)$.
We claim that
$$
(\psi-\delta_0-s\lambda)+\sum_{S \in [n]_2^{n-m}}\epsilon_{S} \delta_{0,S}
$$
has positive degree on any one-parameter family of $m$-stable curves $(f:\C \rightarrow B, \sigman)$. The proof is essentially identical to the proof of Proposition \ref{P:Nefness}, but we will indicate how the proof needs to be modified at each step.

\begin{reductionI}
We may assume that the generic fiber of $\C$ has no disconnecting nodes.
\end{reductionI}
\begin{proof}
As in the proof of Proposition \ref{P:Nefness}, we decompose $\C=\E \cup \R_1 \cup \ldots \cup R_k,$ where $\E \rightarrow B$ is a family of $m$-stable curves whose the general fiber has no disconnecting nodes, and each $\R_i \rightarrow B$ is a stable family of genus zero curves. By condition (2) in our choice of $c$,
$(\psi-\delta_0)+\sum_{S \in [n]_2^{n-m}}\epsilon_{S} \delta_{0,S}$ has positive degree on each of the families $\R_i \rightarrow B$. Arguing as in Proposition \ref{P:Nefness}, we see that it is sufficient to prove the nefness of $(\psi-\delta_0-s\lambda)+\sum_{S \in [n]_2^{n-m}}\epsilon_{S} \delta_{0,S}$ on $\E \rightarrow B$.
\end{proof}

\begin{reductionII}
We may assume that $\lambda.B<0$.
\end{reductionII}
\begin{proof}
Using the relations in Proposition \ref{P:PicardGroup}, we have
$$
(\psi-\delta_0-s\lambda).B+\sum_{S \subset [n]_2^{n-m}} \epsilon_{S} \delta_{0,S}=\sum_{S \subset [n]_2^{n-m}}(|S|-1+\epsilon_{S})\delta_{0,S}.B+(n-s)\lambda.B.
$$
Since $|S| \geq 2$ and $|\epsilon_{S}|<1$, the coefficients $(|S|-1+\epsilon_{S})$ are positive. Arguing precisely as in the proof of Proposition \ref{P:Nefness}, we may assume that $\lambda.B<0$.
\end{proof}

\begin{reductionIII}
We may assume the generic fiber of $\C$ contains an elliptic $l$-fold point, for some $l \geq 1$.
\end{reductionIII}
\begin{proof}
Follows precisely as in the proof of Proposition \ref{P:Nefness}
\end{proof}

Now suppose that \emph{every} fiber of $\C$ has an elliptic $l$-fold point. Then Lemma \ref{L:StableComparison} gives
$$(\psi-\delta_0-s\lambda).B+\sum_{S \subset [n]_2^{n-m}} \epsilon_{S} \delta_{0,S}.B=(\psi^s-\delta^{s}_0).B+\sum_{S \subset [n]_2^{n-m}} \epsilon_{S} \delta_{0,S}^s.B+(l-s)\lambda.B.$$
Our choice of $c$ ensures that $(\psi^s-\delta^{s}_0).B+\sum_{S \subset [n]_2^{n-m}} \epsilon_{S} \delta_{0,S}^s.B$ is positive, i.e. 
$$(\psi-\delta_0-s\lambda).B+\sum_{S \subset [n]_2^{n-m}} \epsilon_{S} \delta_{0,S}.B>(l-s)\lambda.B.$$
Since $l \leq m <s$ and $\lambda.B<0$, the total intersection number is positive.

It remains to consider the possibility that there is a finite set of points $b_1, \ldots, b_t \in B$, where the fibers of $\C$ acquire elliptic $k$-fold points with $k>l$. Since the restriction of $f$ to $B-\{b_1, \ldots, b_t\}$ is an $l$-stable curve, we have a classifying map $c_l:B \rightarrow \S_{1,n}(l)$, and we set
\begin{align*}
\lambda^l.B:=&\deg_{B}c_l^*\lambda,\\
\psi^l.B:=&\deg_Bc_l^*\psi,\\
\delta_{0,S}^l.B:=&\deg_Bc_l^*\delta_{0,S}.\\
\end{align*}
In the preceding paragraph, we saw that
$$(\psi-\delta_0-s\lambda)^l.B+\sum_{S \subset [n]_2^{n-m}} \epsilon_{S} \delta_{0,S}^l.B \geq 0,$$
so it suffices to show that
$$
(\psi-\delta_0-s\lambda).B-(\psi-\delta_0-s\lambda)^l.B+\sum_{S \subset [n]_2^{n-m}} \epsilon_{S} (\delta_{0,S}.B-\delta_{0,S}^l.B) \geq 0.
$$
The proof of Proposition \ref{P:Nefness} shows that 
\begin{align*}
(\psi-\delta_0-s\lambda).B-(\psi-\delta_0-s\lambda)^l.B>(l-s+1)\sum_{i=1}^{t}k_i \geq \sum_{i=1}^{t}k_i,
\end{align*}
where it takes $k_i$ blow-ups/contractions to transform the $l$-stable fiber over $b_i$ into the $m$-stable fiber.

On the other hand, it is easy to see that
$
\delta_{0}.B-\delta_{0}^l.B \geq -m\sum_{i=1}^{t}k_i,
$
since each of the $k_i$ contractions used to transform the $l$-stable fiber over $b_i$ into the $m$-stable fiber over $b_i$ absorbs no more than $m$ nodes. Thus, we obtain
$$\sum_{S \subset [n]_2^{n-m}} \epsilon_{S}(\delta_{0,S}.B-\delta_{0,S}^l.B) \geq -cm\sum_{i=1}^{t}k_i \geq (m-s) \sum_{i=1}^{t}k_i>-\sum_{i=1}^{t}k_i,$$
where $-cm \geq (m-s)$ follows from condition (1) in our choice of $c$.
Combining the previous two equations, we obtain
$$
(\psi-\delta_0-s\lambda).B-(\psi-\delta_0-s\lambda)^l.B+\sum_{S \subset [n]_2^{n-m}} \epsilon_{S} (\delta_{0,S}.B-\delta_{0,S}^l.B) \geq 0,
$$
as desired.
\end{proof}

\begin{corollary}\label{C:TestCurves}
For any $n \geq 0$, there exists a family of $n$-pointed stable curves $(\pi: \C \rightarrow B, \sigman)$ over a smooth complete curve $B$ such that the generic fiber of $\pi$ is smooth and the only singular fibers of $\pi$ are irreducible nodal curves.
\end{corollary}
\begin{proof}
Since $\M_{1,n}(n-1)$ is projective, a general complete-intersection curve $B \subset \M_{1,n}(n-1)$ will not intersect the codimension-two locus $\bigcup_{l \geq 1} \E_{l}$. The induced family $(\C \rightarrow B, \sigman)$ of $(n-1)$-stable curves has no elliptic $l$-fold points and is therefore stable. Since the only boundary divisor of $\M_{1,n}(n-1)$ is $\Delta_{irr}$, the only singular fibers of $\C \rightarrow B$ will be irreducible nodal.
\end{proof}

\begin{corollary}\label{C:MainResult}
Given $s \in \Q$ and $m, n \in \mathbb{N}$ satisfying $m<n$, we have
\begin{enumerate}
\item[]
\item $D(s)$ is big iff $s \in (12-n, \infty)$\\
\item $\M_{1,n}^s=
\begin{cases}
\M_{1,n} &\text{ iff } s \in (11, \infty)\\
\M_{1,n}(1) & \text{ iff } s \in (10,11]\\
\M_{1,n}(m)^* &\text{ iff } s \in (11-m,12-m)\text{ and $m \in \{2, \ldots, n-2\}$}\\
\M_{1,n}(n-1)^* & \text{ iff } s \in (12-n,13-n]\\
\end{cases}
$
\end{enumerate}
\end{corollary}
\begin{proof}
Let us prove (2) first. Since $\delta_{irr}=12\lambda$, we have
$$
D(s):=s\lambda+\psi-\delta=(s-12)\lambda+\psi-\delta_0 \in \Pic_{\Q}(\M_{1,n}).
$$
Lemma \ref{L:BasicNefness} implies that $D(s)$ is ample on $\M_{1,n}$ for $s \in (11, \infty)$ since it lies in the interior of the convex hull of $\lambda$ and $\psi-\delta_0-\lambda$. This implies $\M_{1,n}^{s}=\M_{1,n}$ for $s \in (11, \infty)$.

Next, let us show that $s \in (11-m, 12-m)$ implies $
R(\M_{1,n}, D(s))=R(\NM{n}{m}, \phi_*D(s)).
$ for all $m \in \{1, \ldots, n-1\}$. Consider the birational contraction $\phi: \M_{1,n} \dashrightarrow \M_{1,n}(m)^*.$ By Proposition \ref{P:Discrepancy},
$
R(\M_{1,n}, D(s))=R(\NM{n}{m}, \phi_*D(s))
$
for all $s \in (11-m, 12-m)$. Using Proposition \ref{P:PushPull}, we have
$$
\phi_*D(s)=(s-12)\lambda + \psi-\delta_0 \in \Pic(\M_{1,n}(m)^*).
$$
Thus, Proposition \ref{P:Ampleness} implies that $\phi_*D(s)$ is ample on $\M_{1,n}(m)^*$ if $s \in (11-m, 12-m)$. It follows that
$$
R(\M_{1,n}, D(s))=R(\NM{n}{m}, \phi_*D(s))=\NM{n}{m},
$$
as desired. Finally, the fact that $\M_{1,n}^{12-m}=\M_{1,n}(m)$ iff $m=1$ or $m=n-1$ is a formal consequence of the fact that the rational map $\M_{1,n}(m-1) \dashrightarrow \M_{1,n}(m)$ is regular iff $m=1$ or $m=n-1$ (Corollary \ref{C:Regularity}).

It remains to prove (1). It is clear that $D(s)$ is big for $s > 12-n$ since $D(s)$ becomes ample on a suitable birational model of $\M_{1,n}$ (for all but finitely many values of $s$). On the other hand, if $s=12-n$, then we may consider $\phi: \M_{1,n} \dashrightarrow \M_{1,n}(n-1)$, and one easily checks that $\phi_*D(s) \equiv 0 \in N^1(\M_{1,n}(m))$. Thus, Proposition \ref{P:Discrepancy} implies that $H^0(\M_{1,n}, mD(s))=H^0(\M_{1,n}(m)$, $mD(s)) \leq 1$ for all $m \geq 0$, so $D(s)$ is not big.
\end{proof}
\subsection{$\S_{1,n}(m)$ is singular for $m \geq 6$}\label{S:Singularities}
In this section, we use intersection theory to prove that $\S_{1,n}(m)$ is singular for $m \geq 6$. By Lemma \ref{L:FormalSmoothness}, the singularities of $\S_{1,n}(m)$ depend only on $m$, so it is sufficient to prove that $\S_{1,7}(6)$ is singuar. The main idea is to study the discrepancies of the exceptional divisors of the regular birational contraction $\M_{1,7}(5) \rightarrow \M_{1,7}(6)$.
\begin{lemma}\label{L:CanonicalDivisor}
\begin{align*}
K_{\M_{1,n}} &\equiv \frac{n-11}{12}\Delta_{irr}+\sum_{S \subset [n]_{2}^{m}}(|S|-2)\Delta_{0,S}-\Delta_{0,[n]}
\end{align*}
\end{lemma}
\begin{proof}
A standard application of the Grothendieck-Riemann-Roch \cite[Section 3E]{HarMor} shows that
$$
K_{\S_{1,n}} = 13\lambda-2\delta+\psi \in \Pic(\S_{1,n})
$$
Using the relations in $\Pic(\S_{1,n})$ to rewrite this in terms of boundary divisors (Proposition 3.1), we have
$$
K_{\S_{1,n}} \equiv \frac{n-11}{12}\Delta_{irr}+\sum_{S \subset [n]_{2}^{m}}(|S|-2)\Delta_{0,S}$$
Finally, the map $\S_{1,n} \rightarrow \M_{1,n}$ is ramified along the divisor $\Delta_{0,[n]}$, so we obtain
$$
K_{\M_{1,n}} \equiv \frac{n-11}{12}\Delta_{irr}+\sum_{S \subset [n]_{2}^{m}}(|S|-2)\Delta_{0,S} -\Delta_{0,[n]},
$$
as desired.
\end{proof}

The following lemma says that we can detect singularities by studying the discrepancies of birational contractions.
\begin{lemma}\label{L:Discrep}
Suppose $\phi:X \rightarrow Y$ is a birational morphism of normal, projective varieties, such that $\phi(\Exc(\phi))$ is a finite collection of smooth points of $Y$. Then the discrepancy of any exceptional divisor of $\phi$ is at least $\dim Y-1$.
\end{lemma}
\begin{proof}
Since the question is local on $Y$, we may assume that $Y$ is smooth and that $\phi(\Exc(\phi))=p$ is a single point of $Y$. Since the discrepancy of any exceptional divisor $E$ depends only on the behavior of $\phi$ around a generic point of $E$, it is sufficient to prove the lemma after passing to a resolution of singularities of $X$, i.e. we may assume that $X$ is smooth. By the universal property of blow-ups \cite[Proposition 1.43]{Debarre}, $\phi$ factors as
\[
\xymatrix{
X \ar[r]^{\phi_m} & X_{m} \ar[r]^{\epsilon_{m}}& X_{m-1} \ar[r]^{\epsilon_{m-1}} &\cdots  \ar[r]^{\epsilon_{2}} &X_{1} \ar[r]^{\epsilon_1}& Y,
}
\]
where each $\epsilon_{i}$ is a blow-up along a smooth center, and the restriction
$$
\phi_{m}|_{E}:E \rightarrow \phi_{m}(E)
$$
is birational for each $\phi$-exceptional divisor $E$. Thus, for the purpose of computing discrepancies, we may assume that $X=X_{m}$ and $\phi=\epsilon_{m} \circ \cdots \circ \epsilon_{1}$ is a composition of blow-ups along smooth centers. Since $\epsilon_{1}$ is the blow-up of $Y$ at $p$, we have
$$
\epsilon_{1}^{*}K_{Y}=K_{X_1}+(\dim Y-1)E_{1},
$$
where $E_{1}$ is the exceptional divisor of $\epsilon_{1}$. But since any other $\phi$-exceptional divisor $E$ is centered over $E_{1}$, its discrepancy must be at least $(\dim Y-1)$.
\end{proof}

\begin{corollary}\label{C:Singular}
$\SV{n}{m}$ is not smooth when $m  \geq 6$.
\end{corollary}
\begin{proof}
It suffices to prove that $\S_{1,7}(6)$ is not smooth. Suppose, to the contrary, that $\S_{1,7}(6)$ were smooth. Then the coarse moduli space $\M_{1,7}(6)$ would be a normal projective variety. Furthermore, since the finitely many points of $\S_{1,7}(6)$ corresponding to curves with elliptic 6-fold points have no stabilizer,  $\M_{1,7}(6)$ would be smooth at these finitely many points. By Corollary \ref{C:RegularLocus}, the birational map
$$
\phi: \M_{1,7}(5) \rightarrow \M_{1,7}(6)
$$
is regular, with exceptional divisors $\{\Delta_{0,S}: S \subset [7], |S|=2\}$. Furthermore, if $\phi_{m}: \M_{1,n} \dashrightarrow \M_{1,n}(m)$ denotes the natural birational contraction, Lemma \ref{L:CanonicalDivisor} and Proposition \ref{P:PushPull} give
\begin{align*}
K_{\M_{1,7}(5)}&=(\phi_5)_*K_{\M_{1,7}}=\frac{\!\!-4}{12}\Delta_{irr},\\
K_{\M_{1,7}(6)}&=(\phi_6)_*K_{\M_{1,7}}=\frac{\!\!-4}{12}\Delta_{irr}.\\
\end{align*}
Using Proposition \ref{P:PushPull}, we obtain
$$K_{\S_{1,7}(5)}-\phi^*K_{\S_{1,7}(6)}=\frac{\!\!-4}{12}\Delta_{irr}-\phi^*\left(\frac{\!\!-4}{12}\Delta_{irr}\right)=4\sum_{|S|=2}\Delta_{0,S}.
$$
Since $4<6=\dim \S_{1,7}(6)-1$, this contradicts Lemma \ref{L:Discrep}. We conclude that $\S_{1,7}(6)$ must be singular.
\end{proof}

\bibliography{Bib}
\bibliographystyle{alpha}

\end{document}